\documentclass[a4paper]{article}
\usepackage[english]{babel}
\usepackage[utf8]{inputenc}
\usepackage[T1]{fontenc}
\usepackage[all]{xy}
\usepackage{import} 

\usepackage[a4paper,top=3cm,bottom=2cm,left=3cm,right=3cm,marginparwidth=1.75cm]{geometry}
\usepackage[dvipsnames]{xcolor}
\usepackage{xargs}  
\usepackage[colorinlistoftodos,prependcaption,backgroundcolor=green!40,shadow]{todonotes}
\newcommandx{\pozn}[2][1=]{\todo[linecolor=green,backgroundcolor=green!40,bordercolor=black,#1]{#2}}             

\PassOptionsToPackage{hyphens}{url}
\usepackage{hyperref}
\usepackage{breakurl}
\hypersetup{breaklinks=true}
\usepackage{amsmath}
\usepackage{amssymb}	
\usepackage{amsthm}
\usepackage{mathtools}
\usepackage{tikz-cd}
\usepackage{tikz}
\usepackage{tkz-euclide}
\usepackage{graphicx}
\usepackage{stmaryrd}
\usepackage{verbatim}		
\usepackage{bbold}
\usepackage{enumitem}       
\usepackage{stackrel}       
\usepackage{footmisc}       
\usepackage{svg}

\usepackage[affil-sl]{authblk}

\newcommand{\fld}{\ensuremath{\Bbbk}}

\newcommand{\hdeg}[1]{{|#1|}}

\newcommand{\dg}[1]{|{#1}|}
\newcommand{\id}{1}
\newcommand{\ot}{\otimes}

\newcommand{\cCo}{\ensuremath{\mathsf{Cor}}}

\newcommand{\oP}{\mathcal{P}} 
\newcommand{\oQ}{\mathcal{Q}} 
\newcommand{\co}{\mathbf{o}}	
\newcommand{\Fun}{\mathrm{Fun}} 

\newcommand{\card}{{\tt card}} 
\newcommand{\set}[2]{\left\{ #1 \ | \ #2 \right\} }	
\newcommand{\ooo}[2]{\sideset{_{#1}}{_{#2}}{\mathop{\circ}}} 
\newcommand{\oOo}[2]{\sideset{_{#1}}{_{#2}}{\mathop{\bullet}}} 

\renewcommand{\log}[1]{\mathrm{log} \left( #1 \right)}

\theoremstyle{definition}

\newtheorem{theorem}{Theorem}
\newtheorem{definition}[theorem]{Definition}

\newtheorem{lemma}[theorem]{Lemma}
\newtheorem{remark}[theorem]{Remark}
\newtheorem{example}[theorem]{Example}
\newtheorem{corollary}{Corollary}[theorem]

\makeatletter
\def\blfootnote{\xdef\@thefnmark{}\@footnotetext}
\makeatother

\newcommand{\PICQConSum}{
\begin{tikzpicture}[x=0.75pt,y=0.75pt,yscale=-0.75,xscale=0.75]

\draw    (217.73,177.76) .. controls (240.37,137.54) and (269.12,129.21) .. (250.47,66.8) ;
\draw   (250.47,66.8) .. controls (249.5,63.29) and (256.55,57.37) .. (266.24,53.58) .. controls (275.92,49.79) and (284.56,49.56) .. (285.53,53.08) .. controls (286.51,56.59) and (279.45,62.51) .. (269.77,66.3) .. controls (260.08,70.09) and (251.45,70.32) .. (250.47,66.8) -- cycle ;
\draw   (242.3,194.07) .. controls (241.55,197.27) and (235.43,196.22) .. (228.65,191.71) .. controls (221.86,187.21) and (216.97,180.96) .. (217.73,177.76) .. controls (218.48,174.56) and (224.6,175.62) .. (231.38,180.12) .. controls (238.17,184.63) and (243.06,190.87) .. (242.3,194.07) -- cycle ;
\draw    (242.3,194.07) .. controls (257.93,161.87) and (296.25,173.16) .. (304.99,192.56) ;
\draw   (330.15,169.94) .. controls (331.71,172.98) and (327.35,180.5) .. (320.4,186.75) .. controls (313.45,192.99) and (306.55,195.59) .. (304.99,192.56) .. controls (303.42,189.52) and (307.79,182) .. (314.74,175.75) .. controls (321.69,169.51) and (328.58,166.91) .. (330.15,169.94) -- cycle ;
\draw    (330.15,169.94) .. controls (303.83,137.67) and (295.91,120.51) .. (285.53,53.08) ;
\draw   (106.06,178.79) .. controls (106.37,175.36) and (115.63,173.39) .. (126.74,174.39) .. controls (137.85,175.39) and (146.61,178.98) .. (146.3,182.41) .. controls (145.99,185.84) and (136.73,187.81) .. (125.62,186.81) .. controls (114.51,185.81) and (105.75,182.22) .. (106.06,178.79) -- cycle ;
\draw    (106.06,178.79) .. controls (94.89,130.05) and (76.44,137.97) .. (74.51,105.37) ;
\draw    (146.3,182.41) .. controls (135.13,133.68) and (172.32,131.37) .. (161.15,82.63) ;
\draw    (74.51,105.37) .. controls (73.6,55.72) and (140.83,25.67) .. (161.15,82.63) ;
\draw    (115.27,76.15) .. controls (118.65,78.17) and (120.87,81.9) .. (122,86.51) .. controls (125.07,98.95) and (120.2,117.82) .. (108.64,126.8) ;
\draw    (122,86.51) .. controls (104.24,87.75) and (105.15,108.52) .. (115.53,118.15) ;
\draw  [dash pattern={on 4.5pt off 4.5pt}]  (149.81,112.22) .. controls (195.93,122.88) and (217.93,143.87) .. (256.02,141.25) ;
\draw  [dash pattern={on 0.84pt off 2.51pt}] (258.59,141.56) .. controls (260.17,143.89) and (257.29,150.56) .. (252.15,156.46) .. controls (247.02,162.37) and (241.57,165.27) .. (239.98,162.94) .. controls (238.4,160.62) and (241.28,153.95) .. (246.42,148.04) .. controls (251.56,142.14) and (257,139.24) .. (258.59,141.56) -- cycle ;
\draw  [dash pattern={on 0.84pt off 2.51pt}] (149.81,112.22) .. controls (151.96,114.07) and (150.99,121.09) .. (147.66,127.9) .. controls (144.32,134.71) and (139.88,138.73) .. (137.74,136.88) .. controls (135.6,135.03) and (136.56,128.01) .. (139.9,121.2) .. controls (143.23,114.39) and (147.67,110.37) .. (149.81,112.22) -- cycle ;
\draw  [dash pattern={on 4.5pt off 4.5pt}]  (137.23,136.44) .. controls (161.93,135.87) and (191.93,157.87) .. (239.98,162.94) ;
\draw  [dash pattern={on 0.84pt off 2.51pt}] (281.16,74.85) .. controls (284.51,75.53) and (287.59,82.49) .. (288.03,90.38) .. controls (288.47,98.27) and (286.12,104.11) .. (282.77,103.43) .. controls (279.41,102.75) and (276.34,95.8) .. (275.9,87.91) .. controls (275.45,80.02) and (277.81,74.17) .. (281.16,74.85) -- cycle ;
\draw  [dash pattern={on 0.84pt off 2.51pt}] (287.72,132.85) .. controls (290.6,132.11) and (295.85,136.71) .. (299.44,143.13) .. controls (303.04,149.56) and (303.62,155.37) .. (300.74,156.12) .. controls (297.87,156.86) and (292.62,152.26) .. (289.02,145.83) .. controls (285.43,139.41) and (284.85,133.6) .. (287.72,132.85) -- cycle ;
\draw  [dash pattern={on 4.5pt off 4.5pt}]  (281.16,74.85) .. controls (334.93,36.87) and (406.93,138.87) .. (300.74,156.12) ;
\draw  [dash pattern={on 4.5pt off 4.5pt}]  (282.77,103.43) .. controls (327.93,63.87) and (358.83,145.67) .. (287.72,132.85) ;

\node at (205,108) {$\#_2$};
\node at (375,110){$\#_1$};

\end{tikzpicture}}

\newcommand{\PICAxCSfiveA}{
\begin{tikzpicture}[x=0.75pt,y=0.75pt,yscale=-0.75,xscale=0.75]

\draw   (9.54,95.36) .. controls (9.26,90.71) and (17.4,86.44) .. (27.73,85.82) .. controls (38.05,85.21) and (46.65,88.49) .. (46.93,93.15) .. controls (47.21,97.8) and (39.06,102.07) .. (28.74,102.69) .. controls (18.41,103.3) and (9.82,100.02) .. (9.54,95.36) -- cycle ;
\draw   (86.73,93.6) .. controls (90.71,91.13) and (98.37,96.22) .. (103.82,104.96) .. controls (109.28,113.69) and (110.48,122.77) .. (106.5,125.23) .. controls (102.51,127.69) and (94.86,122.61) .. (89.4,113.87) .. controls (83.94,105.14) and (82.74,96.06) .. (86.73,93.6) -- cycle ;
\draw   (10.5,164.47) .. controls (11.03,159.83) and (19.79,157.03) .. (30.07,158.19) .. controls (40.35,159.36) and (48.24,164.07) .. (47.71,168.7) .. controls (47.18,173.34) and (38.41,176.15) .. (28.14,174.98) .. controls (17.86,173.81) and (9.96,169.1) .. (10.5,164.47) -- cycle ;
\draw   (150.08,126.86) .. controls (146.07,124.42) and (147.18,115.33) .. (152.55,106.55) .. controls (157.92,97.77) and (165.52,92.63) .. (169.53,95.07) .. controls (173.54,97.5) and (172.43,106.59) .. (167.06,115.37) .. controls (161.69,124.15) and (154.08,129.29) .. (150.08,126.86) -- cycle ;
\draw    (9.54,95.36) .. controls (17.34,118.08) and (26.93,140.17) .. (10.5,164.47) ;
\draw    (47.71,168.7) .. controls (48.15,147.91) and (68.3,133.67) .. (106.5,125.23) ;
\draw    (46.93,93.15) .. controls (54.73,115.86) and (71.79,105.95) .. (86.73,93.6) ;
\draw    (150.08,126.86) .. controls (185,158.25) and (236.92,196.49) .. (255.32,149.29) ;
\draw    (169.53,95.07) .. controls (211.75,109.52) and (268.75,102.19) .. (255.32,149.29) ;
\draw    (190.27,123.84) .. controls (193.03,128.99) and (216.44,146.82) .. (234.98,135.34) ;
\draw    (198.75,131.29) .. controls (204.54,127.82) and (215.51,125.17) .. (226.55,138.61) ;
\draw  [dash pattern={on 0.84pt off 2.51pt}] (43.2,140.09) .. controls (41.21,137.38) and (44.33,131.7) .. (50.17,127.41) .. controls (56,123.11) and (62.35,121.82) .. (64.33,124.53) .. controls (66.32,127.23) and (63.2,132.91) .. (57.36,137.21) .. controls (51.52,141.5) and (45.18,142.79) .. (43.2,140.09) -- cycle ;
\draw  [dash pattern={on 0.84pt off 2.51pt}] (190.93,140.37) .. controls (193.03,137.75) and (199.32,139.3) .. (204.98,143.83) .. controls (210.64,148.36) and (213.53,154.15) .. (211.42,156.77) .. controls (209.32,159.39) and (203.03,157.84) .. (197.37,153.31) .. controls (191.71,148.79) and (188.83,142.99) .. (190.93,140.37) -- cycle ;
\draw  [dash pattern={on 4.5pt off 4.5pt}]  (64.33,124.53) .. controls (110.5,158.75) and (150.93,170.37) .. (190.93,140.37) ;
\draw  [dash pattern={on 4.5pt off 4.5pt}]  (43.2,140.09) .. controls (114.5,191.75) and (171.42,186.77) .. (211.42,156.77) ;
\draw   (298.04,93.86) .. controls (297.76,89.21) and (305.9,84.94) .. (316.23,84.32) .. controls (326.55,83.71) and (335.15,86.99) .. (335.43,91.65) .. controls (335.71,96.3) and (327.56,100.57) .. (317.24,101.19) .. controls (306.91,101.8) and (298.32,98.52) .. (298.04,93.86) -- cycle ;
\draw   (375.23,92.1) .. controls (379.21,89.63) and (386.87,94.72) .. (392.32,103.46) .. controls (397.78,112.19) and (398.98,121.27) .. (395,123.73) .. controls (391.01,126.19) and (383.36,121.11) .. (377.9,112.37) .. controls (372.44,103.64) and (371.24,94.56) .. (375.23,92.1) -- cycle ;
\draw   (299,162.97) .. controls (299.53,158.33) and (308.29,155.53) .. (318.57,156.69) .. controls (328.85,157.86) and (336.74,162.57) .. (336.21,167.2) .. controls (335.68,171.84) and (326.91,174.65) .. (316.64,173.48) .. controls (306.36,172.31) and (298.46,167.6) .. (299,162.97) -- cycle ;
\draw   (438.58,125.36) .. controls (434.57,122.92) and (435.68,113.83) .. (441.05,105.05) .. controls (446.42,96.27) and (454.02,91.13) .. (458.03,93.57) .. controls (462.04,96) and (460.93,105.09) .. (455.56,113.87) .. controls (450.19,122.65) and (442.58,127.79) .. (438.58,125.36) -- cycle ;
\draw    (298.04,93.86) .. controls (305.84,116.58) and (315.43,138.67) .. (299,162.97) ;
\draw    (336.21,167.2) .. controls (336.65,146.41) and (356.8,132.17) .. (395,123.73) ;
\draw    (335.43,91.65) .. controls (343.23,114.36) and (360.29,104.45) .. (375.23,92.1) ;
\draw    (438.58,125.36) .. controls (473.5,156.75) and (525.42,194.99) .. (543.82,147.79) ;
\draw    (458.03,93.57) .. controls (500.25,108.02) and (557.25,100.69) .. (543.82,147.79) ;
\draw    (478.77,122.34) .. controls (481.53,127.49) and (504.94,145.32) .. (523.48,133.84) ;
\draw    (487.25,129.79) .. controls (493.04,126.32) and (504.01,123.67) .. (515.05,137.11) ;
\draw  [dash pattern={on 0.84pt off 2.51pt}] (331.7,138.59) .. controls (329.71,135.88) and (332.83,130.2) .. (338.67,125.91) .. controls (344.5,121.61) and (350.85,120.32) .. (352.83,123.03) .. controls (354.82,125.73) and (351.7,131.41) .. (345.86,135.71) .. controls (340.02,140) and (333.68,141.29) .. (331.7,138.59) -- cycle ;
\draw  [dash pattern={on 0.84pt off 2.51pt}] (487.93,141.87) .. controls (490.03,139.25) and (496.32,140.8) .. (501.98,145.33) .. controls (507.64,149.86) and (510.53,155.65) .. (508.42,158.27) .. controls (506.32,160.89) and (500.03,159.34) .. (494.37,154.81) .. controls (488.71,150.29) and (485.83,144.49) .. (487.93,141.87) -- cycle ;
\draw  [dash pattern={on 4.5pt off 4.5pt}]  (352.83,123.03) .. controls (399,157.25) and (447.93,171.87) .. (487.93,141.87) ;
\draw  [dash pattern={on 4.5pt off 4.5pt}]  (331.7,138.59) .. controls (403,190.25) and (468.42,188.27) .. (508.42,158.27) ;
\draw  [dash pattern={on 4.5pt off 4.5pt}]  (9.54,95.36) .. controls (0.5,-4.25) and (173,22.75) .. (106.5,125.23) ;
\draw  [dash pattern={on 4.5pt off 4.5pt}]  (46.93,93.15) .. controls (24.5,48.75) and (111.73,59.1) .. (86.73,93.6) ;
\draw  [dash pattern={on 4.5pt off 4.5pt}]  (395,123.73) .. controls (411,118.75) and (417,113.25) .. (438.58,125.36) ;
\draw  [dash pattern={on 4.5pt off 4.5pt}]  (375.23,92.1) .. controls (397.5,76.75) and (441.5,79.75) .. (458.03,93.57) ;

\end{tikzpicture}
}

\newcommand{\PICBDalg}{
\begin{tikzpicture}[x=0.75pt,y=0.75pt,yscale=-0.84,xscale=0.84]

\draw   (221.36,58.93) .. controls (221.24,56.88) and (225.17,54.99) .. (230.16,54.71) .. controls (235.14,54.43) and (239.29,55.86) .. (239.42,57.9) .. controls (239.55,59.95) and (235.61,61.84) .. (230.62,62.12) .. controls (225.64,62.4) and (221.49,60.97) .. (221.36,58.93) -- cycle ;
\draw   (264.41,57.55) .. controls (266.25,56.46) and (269.81,58.73) .. (272.36,62.62) .. controls (274.9,66.51) and (275.47,70.54) .. (273.63,71.63) .. controls (271.78,72.72) and (268.22,70.45) .. (265.67,66.56) .. controls (263.13,62.67) and (262.56,58.63) .. (264.41,57.55) -- cycle ;
\draw   (225.2,95.91) .. controls (225.41,94.24) and (229.55,93.33) .. (234.46,93.86) .. controls (239.38,94.4) and (243.2,96.18) .. (243,97.84) .. controls (242.8,99.51) and (238.65,100.42) .. (233.74,99.89) .. controls (228.82,99.35) and (225,97.57) .. (225.2,95.91) -- cycle ;
\draw   (291.34,60.14) .. controls (290.72,58.49) and (294.02,55.86) .. (298.71,54.26) .. controls (303.41,52.65) and (307.72,52.69) .. (308.35,54.34) .. controls (308.97,55.98) and (305.67,58.62) .. (300.98,60.22) .. controls (296.28,61.82) and (291.97,61.79) .. (291.34,60.14) -- cycle ;
\draw    (221.31,58.06) .. controls (221.57,66.13) and (236.06,71.56) .. (225.2,95.91) ;
\draw    (243,97.84) .. controls (243.24,86.56) and (251.92,76.21) .. (273.63,71.63) ;
\draw    (239.42,57.91) .. controls (243.85,70.23) and (256.68,63.8) .. (265.16,57.1) ;
\draw    (291.34,60.14) .. controls (295.49,84.5) and (305.75,116.05) .. (328.84,104) ;
\draw    (350,78.51) .. controls (337.54,86.95) and (335.64,100.14) .. (328.84,104) ;
\draw    (307.68,69.44) .. controls (307.09,72.46) and (310.99,87.47) .. (322.93,88.88) ;
\draw    (308.75,75.28) .. controls (312.44,75.77) and (318.08,78.3) .. (318.17,87.44) ;
\draw  [dash pattern={on 0.84pt off 2.51pt}] (240.68,82.74) .. controls (239.85,81.66) and (241.59,79.08) .. (244.58,76.98) .. controls (247.56,74.89) and (250.65,74.06) .. (251.49,75.15) .. controls (252.32,76.23) and (250.58,78.81) .. (247.59,80.91) .. controls (244.61,83) and (241.52,83.83) .. (240.68,82.74) -- cycle ;
\draw  [dash pattern={on 0.84pt off 2.51pt}] (305.21,88.63) .. controls (306.74,87.8) and (309.38,89.47) .. (311.1,92.36) .. controls (312.81,95.25) and (312.97,98.26) .. (311.43,99.09) .. controls (309.9,99.92) and (307.27,98.24) .. (305.55,95.36) .. controls (303.83,92.47) and (303.68,89.46) .. (305.21,88.63) -- cycle ;
\draw  [dash pattern={on 4.5pt off 4.5pt}]  (221.31,58.06) .. controls (221.01,12.21) and (308.1,32.81) .. (273.63,71.63) ;
\draw  [dash pattern={on 4.5pt off 4.5pt}]  (239.42,57.91) .. controls (230.09,34.94) and (276.67,42.98) .. (265.16,57.1) ;
\draw   (333.61,70.34) .. controls (334.71,68.3) and (339.27,68.48) .. (343.79,70.73) .. controls (348.32,72.99) and (351.1,76.47) .. (350,78.51) .. controls (348.89,80.55) and (344.33,80.37) .. (339.81,78.12) .. controls (335.28,75.86) and (332.5,72.38) .. (333.61,70.34) -- cycle ;
\draw    (333.75,70.07) .. controls (321.03,75.08) and (310.42,61.24) .. (308.35,54.34) ;
\draw  [dash pattern={on 4.5pt off 4.5pt}]  (240.68,82.74) .. controls (257.65,101.39) and (288.71,115.36) .. (311.44,99.09) ;
\draw  [dash pattern={on 4.5pt off 4.5pt}]  (251.49,75.15) .. controls (268.45,93.8) and (282.11,100.31) .. (305.21,88.63) ;
\draw   (386.01,60.74) .. controls (385.88,58.7) and (389.82,56.81) .. (394.8,56.53) .. controls (399.79,56.25) and (403.93,57.68) .. (404.06,59.72) .. controls (404.19,61.77) and (400.25,63.66) .. (395.27,63.94) .. controls (390.28,64.22) and (386.13,62.79) .. (386.01,60.74) -- cycle ;
\draw   (429.05,59.36) .. controls (430.89,58.27) and (434.45,60.55) .. (437,64.44) .. controls (439.55,68.32) and (440.11,72.36) .. (438.27,73.45) .. controls (436.42,74.54) and (432.86,72.27) .. (430.32,68.38) .. controls (427.77,64.49) and (427.2,60.45) .. (429.05,59.36) -- cycle ;
\draw   (389.85,97.73) .. controls (390.05,96.06) and (394.19,95.15) .. (399.11,95.68) .. controls (404.02,96.21) and (407.84,97.99) .. (407.64,99.66) .. controls (407.44,101.32) and (403.29,102.24) .. (398.38,101.71) .. controls (393.47,101.17) and (389.65,99.39) .. (389.85,97.73) -- cycle ;
\draw   (455.98,61.96) .. controls (455.36,60.31) and (458.66,57.68) .. (463.35,56.07) .. controls (468.05,54.47) and (472.36,54.51) .. (472.99,56.16) .. controls (473.61,57.8) and (470.31,60.44) .. (465.62,62.04) .. controls (460.92,63.64) and (456.61,63.6) .. (455.98,61.96) -- cycle ;
\draw    (385.95,59.88) .. controls (386.21,67.95) and (400.7,73.38) .. (389.85,97.73) ;
\draw    (407.64,99.66) .. controls (407.89,88.38) and (416.57,78.02) .. (438.27,73.45) ;
\draw    (404.06,59.72) .. controls (408.49,72.04) and (421.32,65.62) .. (429.81,58.92) ;
\draw    (455.98,61.96) .. controls (460.14,86.32) and (470.39,117.87) .. (493.48,105.82) ;
\draw    (514.64,80.32) .. controls (502.18,88.77) and (500.29,101.95) .. (493.48,105.82) ;
\draw    (472.32,71.25) .. controls (471.73,74.27) and (475.63,89.29) .. (487.57,90.7) ;
\draw    (473.39,77.1) .. controls (477.09,77.58) and (482.72,80.12) .. (482.82,89.26) ;
\draw  [dash pattern={on 0.84pt off 2.51pt}] (405.33,84.56) .. controls (404.49,83.48) and (406.24,80.9) .. (409.22,78.8) .. controls (412.2,76.7) and (415.3,75.88) .. (416.13,76.97) .. controls (416.96,78.05) and (415.22,80.63) .. (412.24,82.72) .. controls (409.25,84.82) and (406.16,85.64) .. (405.33,84.56) -- cycle ;
\draw  [dash pattern={on 0.84pt off 2.51pt}] (469.85,90.45) .. controls (471.38,89.62) and (474.02,91.29) .. (475.74,94.18) .. controls (477.46,97.07) and (477.61,100.08) .. (476.08,100.91) .. controls (474.54,101.73) and (471.91,100.06) .. (470.19,97.17) .. controls (468.47,94.28) and (468.32,91.27) .. (469.85,90.45) -- cycle ;
\draw  [dash pattern={on 4.5pt off 4.5pt}]  (455.98,61.96) .. controls (442.31,19.24) and (544.04,26.75) .. (514.64,80.32) ;
\draw  [dash pattern={on 4.5pt off 4.5pt}]  (472.99,56.16) .. controls (469.47,36.68) and (522.2,51.71) .. (498.4,71.89) ;
\draw   (498.25,72.16) .. controls (499.35,70.12) and (503.91,70.29) .. (508.44,72.55) .. controls (512.96,74.8) and (515.74,78.28) .. (514.64,80.32) .. controls (513.54,82.36) and (508.97,82.19) .. (504.45,79.93) .. controls (499.92,77.68) and (497.15,74.2) .. (498.25,72.16) -- cycle ;
\draw    (498.4,71.89) .. controls (485.67,76.9) and (475.07,63.06) .. (472.99,56.16) ;
\draw  [dash pattern={on 4.5pt off 4.5pt}]  (405.33,84.56) .. controls (422.29,103.21) and (453.35,117.18) .. (476.08,100.91) ;
\draw  [dash pattern={on 4.5pt off 4.5pt}]  (416.13,76.97) .. controls (433.1,95.61) and (446.75,102.12) .. (469.85,90.45) ;
\draw   (563.13,65.47) .. controls (563,63.42) and (566.94,61.54) .. (571.92,61.25) .. controls (576.91,60.97) and (581.05,62.4) .. (581.18,64.45) .. controls (581.31,66.49) and (577.37,68.38) .. (572.39,68.66) .. controls (567.4,68.95) and (563.26,67.52) .. (563.13,65.47) -- cycle ;
\draw   (606.17,64.09) .. controls (608.01,63) and (611.57,65.27) .. (614.12,69.16) .. controls (616.67,73.05) and (617.24,77.09) .. (615.39,78.17) .. controls (613.54,79.26) and (609.98,76.99) .. (607.44,73.1) .. controls (604.89,69.21) and (604.32,65.18) .. (606.17,64.09) -- cycle ;
\draw   (566.97,102.45) .. controls (567.17,100.79) and (571.31,99.87) .. (576.23,100.4) .. controls (581.14,100.94) and (584.96,102.72) .. (584.76,104.38) .. controls (584.56,106.05) and (580.41,106.96) .. (575.5,106.43) .. controls (570.59,105.9) and (566.77,104.11) .. (566.97,102.45) -- cycle ;
\draw   (633.11,66.68) .. controls (632.48,65.04) and (635.78,62.4) .. (640.48,60.8) .. controls (645.17,59.2) and (649.48,59.23) .. (650.11,60.88) .. controls (650.74,62.53) and (647.44,65.16) .. (642.74,66.76) .. controls (638.05,68.37) and (633.73,68.33) .. (633.11,66.68) -- cycle ;
\draw    (563.07,64.6) .. controls (563.34,72.68) and (577.82,78.1) .. (566.97,102.45) ;
\draw    (584.76,104.38) .. controls (585.01,93.11) and (593.69,82.75) .. (615.39,78.17) ;
\draw    (581.18,64.45) .. controls (585.62,76.77) and (598.44,70.34) .. (606.93,63.64) ;
\draw    (633.11,66.68) .. controls (637.26,91.04) and (647.51,122.6) .. (670.6,110.54) ;
\draw    (691.76,85.05) .. controls (679.31,93.49) and (677.41,106.68) .. (670.6,110.54) ;
\draw    (649.44,75.98) .. controls (648.85,79) and (652.75,94.01) .. (664.7,95.42) ;
\draw    (650.51,81.82) .. controls (654.21,82.31) and (659.84,84.84) .. (659.94,93.98) ;
\draw  [dash pattern={on 0.84pt off 2.51pt}] (662.27,106.25) .. controls (661.43,105.16) and (663.18,102.59) .. (666.16,100.49) .. controls (669.14,98.39) and (672.24,97.57) .. (673.07,98.65) .. controls (673.91,99.74) and (672.16,102.32) .. (669.18,104.41) .. controls (666.2,106.51) and (663.1,107.33) .. (662.27,106.25) -- cycle ;
\draw  [dash pattern={on 0.84pt off 2.51pt}] (642.71,85.72) .. controls (644.24,84.9) and (646.88,86.57) .. (648.6,89.46) .. controls (650.32,92.34) and (650.47,95.36) .. (648.94,96.18) .. controls (647.4,97.01) and (644.77,95.34) .. (643.05,92.45) .. controls (641.33,89.56) and (641.18,86.55) .. (642.71,85.72) -- cycle ;
\draw   (675.37,76.89) .. controls (676.47,74.85) and (681.03,75.02) .. (685.56,77.27) .. controls (690.08,79.53) and (692.86,83.01) .. (691.76,85.05) .. controls (690.66,87.09) and (686.1,86.91) .. (681.57,84.66) .. controls (677.04,82.41) and (674.27,78.92) .. (675.37,76.89) -- cycle ;
\draw    (675.52,76.61) .. controls (662.79,81.63) and (652.19,67.79) .. (650.11,60.88) ;
\draw  [dash pattern={on 4.5pt off 4.5pt}]  (648.94,96.18) .. controls (611.5,123.75) and (668,131.75) .. (662.27,106.25) ;
\draw  [dash pattern={on 4.5pt off 4.5pt}]  (642.71,85.72) .. controls (571,125.25) and (689,169.25) .. (673.07,98.65) ;
\draw  [dash pattern={on 4.5pt off 4.5pt}]  (581.18,64.45) .. controls (572.39,45.73) and (623.29,41.86) .. (633.11,66.68) ;
\draw  [dash pattern={on 4.5pt off 4.5pt}]  (650.11,60.88) .. controls (636.07,29.69) and (559.65,22.42) .. (563.07,64.6) ;
\draw   (49.55,59.3) .. controls (48.89,57.35) and (52.2,54.6) .. (56.95,53.15) .. controls (61.71,51.7) and (66.1,52.11) .. (66.76,54.06) .. controls (67.43,56.02) and (64.12,58.77) .. (59.36,60.22) .. controls (54.61,61.66) and (50.22,61.25) .. (49.55,59.3) -- cycle ;
\draw   (90.87,47.83) .. controls (92.37,46.35) and (96.42,47.71) .. (99.91,50.87) .. controls (103.4,54.04) and (105.02,57.82) .. (103.52,59.3) .. controls (102.02,60.79) and (97.97,59.43) .. (94.48,56.27) .. controls (90.99,53.1) and (89.37,49.32) .. (90.87,47.83) -- cycle ;
\draw   (63.04,94.21) .. controls (62.8,92.55) and (66.57,90.69) .. (71.47,90.05) .. controls (76.37,89.41) and (80.54,90.24) .. (80.78,91.9) .. controls (81.03,93.56) and (77.25,95.42) .. (72.36,96.06) .. controls (67.46,96.69) and (63.29,95.87) .. (63.04,94.21) -- cycle ;
\draw   (121.76,56.25) .. controls (121.14,54.61) and (124.44,51.97) .. (129.13,50.37) .. controls (133.83,48.77) and (138.14,48.81) .. (138.77,50.45) .. controls (139.39,52.1) and (136.09,54.73) .. (131.4,56.33) .. controls (126.7,57.94) and (122.39,57.9) .. (121.76,56.25) -- cycle ;
\draw    (49.27,58.47) .. controls (51.66,66.23) and (67.12,68.08) .. (63.04,94.21) ;
\draw    (80.78,91.9) .. controls (78.04,80.92) and (83.71,68.84) .. (103.52,59.3) ;
\draw    (66.77,54.06) .. controls (74.32,64.95) and (85.04,55.71) .. (91.49,47.22) ;
\draw    (121.76,56.25) .. controls (125.91,80.61) and (136.16,112.17) .. (159.26,100.12) ;
\draw    (180.41,74.62) .. controls (167.96,83.07) and (166.06,96.25) .. (159.26,100.12) ;
\draw    (138.1,65.55) .. controls (137.51,68.57) and (141.4,83.58) .. (153.35,84.99) ;
\draw    (139.17,71.39) .. controls (142.86,71.88) and (148.5,74.41) .. (148.59,83.55) ;
\draw  [dash pattern={on 0.84pt off 2.51pt}] (74.55,77.82) .. controls (73.46,76.97) and (74.47,74.06) .. (76.8,71.33) .. controls (79.14,68.59) and (81.92,67.07) .. (83.01,67.92) .. controls (84.1,68.77) and (83.1,71.68) .. (80.76,74.41) .. controls (78.42,77.15) and (75.65,78.67) .. (74.55,77.82) -- cycle ;
\draw  [dash pattern={on 0.84pt off 2.51pt}] (137.43,84.02) .. controls (138.96,83.19) and (141.6,84.86) .. (143.32,87.75) .. controls (145.04,90.64) and (145.19,93.65) .. (143.65,94.48) .. controls (142.12,95.3) and (139.49,93.63) .. (137.77,90.74) .. controls (136.05,87.85) and (135.9,84.84) .. (137.43,84.02) -- cycle ;
\draw  [dash pattern={on 4.5pt off 4.5pt}]  (104.49,58.34) .. controls (110.39,49.58) and (116.56,45.45) .. (121.76,56.25) ;
\draw  [dash pattern={on 4.5pt off 4.5pt}]  (91.49,47.22) .. controls (105.24,28.94) and (136.37,31.37) .. (138.77,50.45) ;
\draw   (164.02,66.46) .. controls (165.13,64.42) and (169.69,64.59) .. (174.21,66.85) .. controls (178.74,69.1) and (181.52,72.58) .. (180.41,74.62) .. controls (179.31,76.66) and (174.75,76.49) .. (170.23,74.23) .. controls (165.7,71.98) and (162.92,68.5) .. (164.02,66.46) -- cycle ;
\draw    (164.17,66.18) .. controls (151.45,71.2) and (140.84,57.36) .. (138.77,50.45) ;
\draw  [dash pattern={on 4.5pt off 4.5pt}]  (74.55,77.82) .. controls (95.91,91.88) and (115.28,105.68) .. (143.65,94.48) ;
\draw  [dash pattern={on 4.5pt off 4.5pt}]  (83.01,67.92) .. controls (104.37,81.99) and (112.45,89.17) .. (137.43,84.02) ;

\draw (70,133.4) node [anchor=north west][inner sep=0.75pt]    {$\sum_{a,b\in X\star Y}$};
\draw (83.33,28) node [anchor=north west][inner sep=0.75pt]    {$\scriptstyle{a}$};
\draw (142,32) node [anchor=north west][inner sep=0.75pt]    {$\scriptstyle{b}$};
\draw (245,133.4) node [anchor=north west][inner sep=0.75pt]    {$\sum_{a,b\in X}$};
\draw (209,41.73) node [anchor=north west][inner sep=0.75pt]    {$\scriptstyle{a}$};
\draw (276,67) node [anchor=north west][inner sep=0.75pt]    {$\scriptstyle{b}$};
\draw (430,133.4) node [anchor=north west][inner sep=0.75pt]    {$\sum_{a,b\in Y}$};
\draw (444,32) node [anchor=north west][inner sep=0.75pt]    {$\scriptstyle{a}$};
\draw (519,70.4) node [anchor=north west][inner sep=0.75pt]    {$\scriptstyle{b}$};
\draw (575,133.4) node [anchor=north west][inner sep=0.75pt]    {$\sum_{a\in X,b\in Y}$};
\draw (551.33,40.5) node [anchor=north west][inner sep=0.75pt]    {$\scriptstyle{a}$};
\draw (654.67,42) node [anchor=north west][inner sep=0.75pt]    {$\scriptstyle{b}$};
\draw (362.33,63.4) node [anchor=north west][inner sep=0.75pt]    {$\pm $};
\draw (539,63.4) node [anchor=north west][inner sep=0.75pt]    {$\pm $};
\draw (188.67,63.4) node [anchor=north west][inner sep=0.75pt]    {$=$};

\end{tikzpicture}

}

\title{Connected sum for modular operads and Beilinson-Drinfeld algebras}
\author{Martin Doubek}
\author{Branislav Jurčo\footnote{\href{mailto:Branislav.Jurco@mff.cuni.cz}{Branislav.Jurco@mff.cuni.cz} \\}}
\author{Lada Peksová\footnote{\href{mailto:lada.peksova@mff.cuni.cz}{lada.peksova@mff.cuni.cz} \\}}
\affil{Mathematical Institute, Faculty of Mathematics and Physics, Charles University, Prague 186 75, Czech Republic}
\author{Ján Pulmann\footnote{\href{mailto:Jan.Pulmann@ed.ac.uk}{Jan.Pulmann@ed.ac.uk}   \\}}
\affil{School of Mathematics, University of Edinburgh, U.K.}

\date{}
\begin{document}
\maketitle

\vspace{-0.8cm}

\vspace*{0.2cm}
\begin{center}
	\LARGE \textit{Dedicated to the memory of Martin Doubek.}
\end{center}
\vspace*{1cm}

\begin{abstract}
Modular operads relevant to string theory can be equipped with an additional structure, coming from the connected sum of surfaces. Motivated by this example, we introduce a notion of connected sum for general modular operads. We show that a connected sum induces a commutative product on the
space of functions associated to the modular operad. Moreover, we combine this product with Barannikov's non-commutative Batalin-Vilkovisky structure present on this space of functions, obtaining a Beilinson-Drinfeld algebra. Finally, we study the quantum master equation using the exponential defined using this commutative product.
\end{abstract}

\section{Introduction}

Batalin-Vilkovisky (BV) formalism \cite{Batalin1981} is a formal integration technique that originated in quantum field theory. Its basic ingredients are an odd, second order differential operator $\Delta$ on the space of functionals and a $\Delta$-closed functional $e^{S/\hbar}$, i.e. a quantum master action. Observables are then computed by taking an integral over a Lagrangian submanifold in the space of fields, weighted by $e^{S/\hbar}$. The closedness of $e^{S/\hbar}$ ensures that the result is independent of the choice of the Lagrangian, generalizing gauge independence.

It was observed by Barannikov \cite{BarannikovModopBV} that one can obtain a similar algebraic setup for any modular operad $\oP$ \cite{GK} in dg vector spaces. Concretely, for any odd symplectic vector space $V$, Barannikov defined a vector space $\Fun_\oP(V)$ equipped with a BV operator and a bracket, giving a non-commutative version of a BV algebra. For $\mathcal P = \mathcal{QC}$, the \emph{quantum closed} modular operad, one recovers the usual BV formalism for $V$ \cite{BarannikovModopBV, markl2001loop, zwiebach-closed}.

The goal of this work is to give a construction of the so-far missing commutative product usually present in the BV formalism. To this end, we introduce the notion of a \emph{connected sum} on a modular operad $\oP$. For a modular operad $\oP$ with such connected sum, we define a commutative product on the space $\Fun_{\oP}(V)$, compatible with the structure introduced by Barannikov. However, in this way we obtain a Beilinson-Drinfeld algebra \cite{BeilinsonDrinfeld, CostelloGwilliam}, a close relative of a BV algebra. In the examples $\mathcal{QC}$ and $\mathcal{QO}$ coming from 2D topology, this structure is induced by the actual connected sum of surface. Thus, the distinction between BV and BD algebras acquires a topological explanations: see Figure \ref{fig:BDaxiom}. We explicitly describe the commutative products in these two examples, getting the expected products of polynomials and cyclic words built from letters in $V^*$.

Instead of the equation $\Delta e^{S/\hbar} = 0$, one usually writes the formally equivalent \emph{quantum master equation}
\[ 2\hbar\Delta S + \{S, S \} = 0, \]
which was also the form used in \cite{BarannikovModopBV}. Using the commutative product, we can now make sense of the exponential $e^{S/\hbar}$, after an appropriate completion. With a simple non-degeneracy condition on the connected sum, we prove that the above quantum master equation is equivalent to $\Delta e^{S/\hbar} =0$.

Commutative products and BV algebra structures on $\Fun_{\oP}(V)$ coming from the \emph{disjoint union} of surfaces were considered Kaufmann, Ward and Z\'{u}\~{n}iga in \cite{Kaufmann-odd}; see also \cite{SenZwiebach, Schwarz, Kaufmann-Feynman}. Connected sum for the modular operad $\mathcal{QO}$ appeared in the recent work of Berger and Kaufmann \cite{BergerKaufmann}. See Remark \ref{remark_literature} for a more detailed review.

\subsection{Notations and conventions}
We consider $\mathbb Z$-graded vector spaces over a field with characteristic zero. The degree of a homogeneous element $v$ is be denoted $\hdeg{v}$. Differentials have degree $+1$.

We denote by $\sqcup$ the disjoint union and $\setminus$ the set difference. By $[n]$, we will denote the set $\lbrace 1, 2, \ldots n \rbrace$. The permutation group of $[n]$ is denoted by $\Sigma_n$. The cardinality of a set $A$ is denoted $\card{A}$.

\subsection{Acknowledgements}
We would like to thank Ralph Kaufmann for answering our questions about Feynman categories and connected sums.

The research of B.J. was supported by grant GAČR EXPRO 19–28628X. The research of J.P. was supported by the NCCR SwissMAP and Postdoc.Mobility 203065 grants of the SNSF. 
\medskip

This paper gives an extended account of some of the results announced by the third author in \cite{Lada}.

\section{Modular operads and the connected sum}\label{2}

Modular operads were introduced by Getzler and Kapranov \cite{GK}. We start by recalling a definition of a modular operad in the spirit of \cite{DJMS, DJM}.

\begin{definition} \label{DEFCorr}
Denote $\cCo$ the \textbf{category of stable corollas}: the objects are pairs $(C,G)$ with $C$ a finite set and $G$ a non-negative half-integer such that the \emph{stability condition} is satisfied: \begin{equation}\label{eq:stability}2(G-1)+\card(C)>0.\end{equation}
Morphisms $(C, G) \to (D, G')$ exist only if $G = G'$, in which case they are bijections $C\xrightarrow{\cong} D$. 
\end{definition}

\begin{remark}
The condition of stability was introduced in the context of modular operads by Getzler and Kapranov, and  its name comes from the theory of moduli spaces of curves.  In our context, the stability condition will ensure convergence of certain formal exponentials, see Corollary \ref{cor:QMEexponential}.
\end{remark}

\begin{definition} \label{DEFModOp}
A \textbf{modular operad} $\oP$ is a functor $\oP$ from $\cCo$ to the category of dg vector spaces (with morphisms of degree 0), together with a collection of degree 0 chain maps
\begin{gather*}
        \set{\ooo{a}{b}\colon\oP(C_1\sqcup\{a\},G_1)\ot\oP(C_2\sqcup\{b\},G_2)\!\to\!\oP(C_1\sqcup C_2,G_1+G_2)}{(C_1,G_1),\!(C_2,G_2)\in\cCo} \,\,\mbox{and}\\
        \set{\ooo{}{ab} = \ooo{}{ba}\colon\oP(C\sqcup\{a,b\},G)\to\oP(C,G+1)}{(C,G)\in\cCo}.
\end{gather*}
These data are required to satisfy the following axioms:
\begin{itemize}
\item (MO1) $(\oP(\rho|_{C_1} \sqcup \sigma|_{C_2})) \ \ooo{a}{b} = \ooo{\rho(a)}{\sigma(b)} \ (\oP(\rho)\ot\oP(\sigma))$,
\item (MO2) $\oP(\rho|_{C}) \ \ooo{}{ab} = \ooo{}{\rho(a)\rho(b)}\oP(\rho)$,
\item (MO3) $\ooo{a}{b} (x\ot y)  = (-1)^{\dg{x}\dg{y}} \ooo{b}{a} (y\ot x)$ \hskip 0.5cm
                        for any $x\in\oP(C_1\sqcup\{a\},G_1)$, $y\in\oP(C_2\sqcup\{b\},G_2)$,
\item (MO4) $\ooo{}{ab} \ \ooo{}{cd} = \ooo{}{cd} \ \ooo{}{ab}$,
\item (MO5) $\ooo{}{ab} \ \ooo{c}{d} = \ooo{}{cd} \ \ooo{a}{b}$,
\item (MO6) $\ooo{a}{b} \ (\ooo{}{cd}\ot\id) = \ooo{}{cd} \ \ooo{a}{b}$\,\, \mbox{and}
\item (MO7) $\ooo{a}{b} \ (\id\ot\ooo{c}{d}) = \ooo{c}{d} \ (\ooo{a}{b}\ot\id)$,
\end{itemize}
whenever the expressions make sense.
\end{definition}

As in \cite{DJM, DJMS}, we also define odd modular operads, which are special cases of \emph{twisted} modular operads of \cite{GK}.

\begin{definition} \label{ODEFModOp}
An \textbf{odd modular operad} is defined similarly as the modular operad with the only exception of the operadic compositions, now denoted as $\oOo{a}{b}$ and $\oOo{}{ab}$, being of degree 1. Correspondingly, the above  axioms (MO4)--(MO7) will get an extra minus sign. See \cite[Definition~4.]{DJM} for details.
\end{definition}

\subsection{Connected sum}
 Let us now define connected sum on a modular operad, motivated by the connected sum operation on surfaces.
\begin{definition}\label{def_connected_sum}
A \textbf{modular operad with a connected sum} is a modular operad $\oP$ equipped with two collection of degree 0 chain maps\footnote{The seemingly strange shift of $G$ by $1$ and $2$ is chosen to match already existing conventions, see Remark \ref{remark_Gshifts} for details.}
\begin{equation}\label{connected_sum_on_two}
\#_2 \colon \oP(C,G) \otimes \oP(C',G') \rightarrow \oP(C\sqcup C', G+G'+1)
\end{equation}
and
\begin{equation}\label{connected_sum_on_one}
\#_1 \colon \oP(C,G)  \rightarrow \oP(C, G+2)
\end{equation}
such that
\begin{itemize}
\item (CS1) $(\oP (\sigma \sqcup\sigma'))\#_2 = \#_2 (\oP(\sigma) \otimes \oP(\sigma'))$, $\oP(\sigma) \#_1 = \#_1 \oP(\sigma)$ for all bijections
$\sigma\colon C\rightarrow D$, $\sigma'\colon C'\rightarrow D'$, 
\item (CS2) $\#_2 \tau=\#_2$, where $\tau$ is the monoidal symmetry (from the category of graded vector spaces),
\item (CS3) $\#_2(1\otimes \#_2)=\#_2(\#_2\otimes 1), \qquad \#_2(\#_1 \otimes 1)= \#_1\#_2 \,$
\item (CS4)  as maps $\oP(C\sqcup \lbrace a,b\rbrace ,G)\rightarrow \oP(C, G+3)$,
$$\ooo{}{ab}\#_1=\#_1 \ooo{}{ab},$$
\item (CS5a) as maps $\oP(C,G)\otimes \oP(C',G')\rightarrow \oP(C\sqcup C'\setminus \lbrace a,b \rbrace, G+G'+2)$, 
$$ \ooo{}{ab} \#_2  = 
  \begin{cases} 
   \#_2 (\ooo{}{ab} \otimes 1) &  \text{if }  a,b \in C \\
   \#_2 (1\otimes \ooo{}{ab}) & \text{if }   a,b \in C'\\
  \#_1 \ooo{a}{b} & \text{if }   a\in C, b\in C'\\
 \#_1 \ooo{b}{a} & \text{if }   b\in C, a\in C',
  \end{cases}$$
\item (CS5b) as maps $\oP(C\sqcup \lbrace a\rbrace,G)\otimes \oP(C'\sqcup \lbrace b \rbrace,G')\rightarrow \oP(C\sqcup C', G+G'+2)$, 
  $$\ooo{a}{b}(\#_1\otimes 1)= \#_1\ooo{a}{b},$$ 
\item (CS6)  as maps $\oP(C\sqcup \lbrace a\rbrace,G)\otimes \oP(C',G')\otimes \oP(C'',G'')\rightarrow \oP(C\sqcup C'\sqcup C''\setminus \lbrace b \rbrace, G+G'+G''+1)$,
$$\ooo{a}{b} (1\otimes \#_2)=
\begin{cases}
\#_2 (\ooo{a}{b}\otimes 1) & \ldots b\in C'\\
\#_2 (1\otimes \ooo{a}{b})(\tau\otimes 1) & \ldots b\in C'',
\end{cases}$$
where the map $(\tau\otimes 1)\colon\oP(C,G)\otimes \oP(C',G')\otimes \oP(C'',G'') \rightarrow \oP(C',G')\otimes \oP(C,G)\otimes \oP(C'',G'') $ switches the first two tensor factors.
\end{itemize}
\end{definition}

\begin{remark}\label{remark_literature}
If one considers the disjoint union of surfaces, instead of the connected sum, its compatibility with the operadic compositions $\ooo{a}{b}$ and $\ooo{}{ab}$ will look similarly to Definition \ref{def_connected_sum}. An important difference will appear in axioms (CS5a): the disjoint union followed by $\ooo{}{ab}$ is just equal to $\ooo{a}{b}$, and there is no analogue of $\#_1$. Such approach to operads, abstracting the disjoint union, was (to our knowledge) first considered  by Schwarz \cite[Sec.~2]{Schwarz}. There, $\nu$ is used for the disjoint sum, $\sigma$ for the self-composition $\ooo{}{ab}$; the composition $\ooo{a}{b}$ can be defined from these two operations.

Later, a similar operation was considered in generality by Borisov and Manin \cite{BorisovManin} and for modular operads by Kaufmann and Ward \cite{Kaufmann-Feynman}, under the name of mergers/horizontal compositions. See e.g. \cite[Eq.~(5.4),(5.5)]{Kaufmann-Feynman} for the disjoint-union-analogue of (CS5a). The commutative product and the resulting BV algebra was studied by Kaufmann, Ward and Z\'{u}\~{n}iga in \cite{Kaufmann-odd}.  A notable precursor in string field theory is the work of Sen and Zwiebach \cite[Sec.~7.1]{SenZwiebach}.

The connected sum of surfaces was considered, for the modular operad $\mathcal{QO}$, by Berger and Kaufmann \cite[Sec.~5.6]{BergerKaufmann}. There, they notice the need for an analogue of (CS5a) \cite[{Sec.~5.6, ``equation 2.9 does not hold''}]{BergerKaufmann} and remark that such connected sums define Feynman categories, functors out of which are equivalent to our modular operads with connected sum.
\end{remark}

Similarly to Definition \ref{def_connected_sum}, we can consider an odd modular operad equipped with a connected sum.
\begin{definition}\label{def_connected_sum_odd}
An \textbf{odd modular operad with a connected sum} is as in Definition \ref{def_connected_sum}, with the black bullet replacing the white one. 
 \end{definition}

Note that $\#_1$ and $\#_2$ are again degree 0 operations. To make the difference between the odd and untwisted cases more explicit, we will write down the axioms (CS5a) and (CS6), evaluated on elements, in both cases 

If $p\in \oP(C,G), p'\in \oP(C',G')$ and $p''\in \oP(C'',G'')$, then in the untwisted case (CS5a),
$$ \ooo{}{ab} (p \, \#_2\,  p') = 
  \begin{cases} 
   (\ooo{}{ab} p)\, \#_2\,  p' & \ldots\  a,b \in C \\
   p \, \#_2\,  ( \ooo{}{ab}p') & \ldots\  a,b \in C' \\
  \#_1 \, ( p \ooo{a}{b} p') & \ldots\  a\in C, b\in C'\\
 \#_1 \, (p \ooo{b}{a} p') & \ldots\  b\in C, a\in C'
  \end{cases}$$
  and in the odd case,
    $$ \oOo{}{ab}(p \,  \#_2 \, p') = 
  \begin{cases} 
   (\oOo{}{ab} p)\, \#_2\,  p' & \ldots\  a,b \in C \\
   p\, \#_2 \, (\oOo{}{ab} p') (-1)^{|p|} & \ldots\  a,b \in C' \\
  \#_1 \, (p \oOo{a}{b} p') & \ldots\  a\in C, b\in C'\\
 \#_1 \, (p \oOo{b}{a} p') & \ldots\  b\in C, a\in C'.
  \end{cases}$$
Concerning (CS6), in untwisted case, we have
$$ p \ooo{a}{b} (p' \, \#_2\,  p'')  = 
  \begin{cases} 
   ( p \ooo{a}{b} p' )\, \#_2\,  p'' & \ldots\  b\in C',\\
   p' \, \#_2\,  (p \ooo{a}{b} p'') & \ldots\  b\in C'',
  \end{cases}$$
 whereas in the odd case,
 $$ p \oOo{}{ab} (p' \, \#_2\,  p'')  = 
  \begin{cases} 
   ( p \oOo{a}{b} p' )\, \#_2\,  p'' & \ldots\  b\in C',\\
   (-1)^{|p||p'|+|p'|} p' \, \#_2\,  (p \oOo{a}{b} p'') & \ldots\  b\in C''.
  \end{cases}$$
  \begin{remark}
In all of the examples we will consider, $\#_1$ will be injective. In this case, the axiom (CS5a) determines the operadic compositions $\ooo{a}{b}$ in terms of $\ooo{}{ab}$, $\#_2$, and $\#_1$, and similarly for $\oOo{a}{b}$. 
  \end{remark}
 
\subsection{Examples of connected sum}
We will now describe a connected sum on two basic modular operads: the quantum closed operad $\mathcal{QC}$ and the quantum open operad $\mathcal{QO}$ (see \cite{DJM, DJMS} for their description as modular operads).

\begin{example}\label{example_QC}
 The quantum closed modular operad $\mathcal{QC}$ is the modular envelope of the cyclic commutative operad, but has an explicit description as follows. For each finite set $C$ and each non-negative integer $g$, we define $\mathcal{QC}(C, 2g + \card(C)/2-1)$ to be a one-dimensional vector space, with generator called $C^g$. This should be seen the homeomorphism class of connected compact oriented surfaces of genus $g$ and with punctures labelled by elements of $C$. The operadic structure corresponds to sewing punctures together. See Remark \ref{remark_Gshifts} for the origin of the definition $G=2g + \card{C}/2-1$.

The connected sum is defined simply as
\begin{eqnarray*}
    C_1^{g_1}\#_2\  C_2^{g_2}&=& (C_1 \sqcup C_2)^{g_1+g_2},\\
    \#_1\, \left( C^{g}\right)&=& C^{g+1}.
\end{eqnarray*}
which increases the ``operadic'' genus $G = 2g + \card(C)/2-1$ by $1$ and $2$, respectively. Geometrically, $\#_2$ corresponds to the connected sum of surfaces and $\#_1$ to adding a handle, as on Figure \ref{fig:QCCS} involving connected sums of $C_1^1$ and $C_2^0$ with $\card (C_1)=1$ and $\card (C_2) = 3$.

\begin{figure}[h]
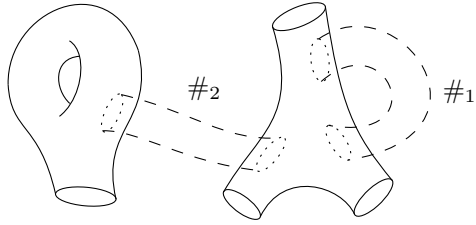

\begin{center}
\PICQConSum
\caption{Connected sum on the quantum closed operad.}
\end{center}
\label{fig:QCCS}
\end{figure}

The axioms of the connected sum are satisfied trivially, but they also have a topological interpretation, as on Figure \ref{fig:QCCSax}.

\begin{figure}[h]
\begin{center}
\PICAxCSfiveA
\caption{Axiom CS5a - cases $\ooo{}{ab} \#_2  = \#_2 (\ooo{}{ab} \otimes 1)$ and $\ooo{}{ab} \#_2  =   \#_1 \ooo{a}{b}$.}
\label{fig:QCCSax}
\end{center}
\end{figure}
\end{example}

\begin{example}\label{example_QO} The quantum open modular operad $\mathcal{QO}$ is defined as follows. A cycle $\co$ in a set $O$ is a (possibly empty) cyclic word made of several distinct elements of $O$.
The components of the modular operad $\mathcal{QO}$ are given as
\begin{gather*}
\mathcal{QO}(O,G) \equiv \mathrm{Span}_{\mathbb{k}} \left\{ \{\co_1, \ldots \co_{b}\}^g \  |\  b,g\in\mathbb{N}_0, \co_i \mathrm{\ cycle\ in\ } O, \bigsqcup_{i=1}^b \co_i=O,  G=2g+b-1\right\}.
\end{gather*}
 Geometrically, the generators correspond to homeomorphism class of a compact oriented surface with genus $g$, $b$ boundaries and punctures  on the boundaries labelled by elements of $O$. The operadic compositions correspond to sewing/self-sewing of surfaces along punctures.

Similarly to the previous example, the modular operad $\mathcal{QO}$ is the modular envelope of the cyclic associative operad $Ass$ by a result of Doubek \cite{DoubekModularEnvelope}.
 
 The connected sum is defined as
\begin{eqnarray*}
\{\co_1, \ldots \co_{b_1}\}^{g_1}\,\#_2\, \{\co'_1, \ldots \co'_{b_2}\}^{g_2} &=&\{\co_1, \ldots \co_{b_1}, \co'_1 \ldots \co'_{b_2} \}^{g_1+g_2},\\
\#_1 \left(\{\co_1, \ldots \co_{b}\}^{g}\right) &=& \{\co_1, \ldots \co_{b}\}^{g+1}.
\end{eqnarray*}
with the same geometric interpretation as in the previous example.
\end{example}

These two examples can be combined as follows.
\begin{example}\label{QOC}
Although we did not introduce colored modular operads, it is easy to see that we can straightforwardly combine the quantum closed operad and quantum open operad into a two-colored quantum open-closed operad $\mathcal{QOC}$ \cite{DJM}. It has components $\mathcal{QOC}(C, O, G)$ generated by homeomorphism classes of surfaces with closed punctures labelled by $C$ and open punctures (lying on the boundary) labelled by $O$. In this case, $G=2g + b + \card(C)/2 -1$
\end{example}

\begin{remark}\label{remark_Gshifts}
Using the above examples, we can explain the dependence of $G$ on $g$ and the shifts of $G$ in Definition \ref{def_connected_sum}. Already for the operad $\mathcal{QO}$, the operadic self-composition $\ooo{}{ab}$ can act on punctures on two different boundary components or on same boundary. Since these two cases change $g$ and $b$ differently, we are led to define $G$ as their linear combination. 

If we require the operations $\ooo{}{ij}$ and $\ooo{i}{j}$ to change $G$ by $+1$ and $0$, then for the more general quantum open-closed operad $\mathcal{QOC}$ we can choose $G = \alpha g + \frac{\alpha}{2} b + \frac{\alpha-1}{2} \card(C) + \frac{\alpha - 2}{4} \card(O) + 1-\alpha$ for any $\alpha$. Moreover, $\#_2$ will increase $G$ by $\alpha-1$ and $\#_1$ by $\alpha$. Our choice of $G$ corresponds to $\alpha = 2$, which was used in \cite{DJM} and ultimately comes from the work of Zwiebach, where he wants $G=0$ for the disc with three open punctures on the boundary \cite[eq.~3.11]{ZwiebachOC}.
\end{remark}

Similarly, there exists colored generalization of these operads coming from string field theory, let us mention the one coming form the type II superstring field theory.
\begin{example}
A four-colored operad relevant relevant to type II superstring field theory was introduced in \cite{JM}. The geometric picture here is based on super Riemannian surfaces with four kinds of punctures corresponding to the four respective sectors $NS-NS$, $NS-R$, $R-NS$ and $R-R$. The geometric representation of the connected sum would remain the same as above.
\end{example}


\subsection{Endomorphism operad and the connected sum}\label{section_endomorphism_operad}
In this section, we will describe our main example of an odd modular operad with a connected sum, the endomorphism operad. Let us first recall the definition of the unordered tensor product and positional derivatives.

\begin{definition}
For a finite set $C$ with $\card(C)=n$ and a vector space $V$, we define the \textbf{unordered tensor product} as
$$ \bigotimes_{ C} V =\left. \bigoplus_{\psi\colon C\rightarrow [n]} V^{\otimes n}\right/ \sim \, .$$ 
where the equivalence relation is given by
$i_\psi(v_1 \otimes \ldots \otimes v_n)\sim i_{\sigma\psi}(\sigma(v_1 \otimes \ldots \otimes v_n))$, where $\sigma \in \Sigma_n$ and $i_\psi\colon V^{\otimes n}\hookrightarrow \bigoplus_{\psi\colon C\rightarrow [n]}
V^{\otimes n}$ is the canonical inclusion into the $\psi$-th summand\footnote{In other words, choosing a linear order of $C$ gives an isomorphism between $\bigotimes_C V$ and $V^{\otimes n}$, with different
isomorphisms related by the corresponding permutation.}. 

The map $V^{\otimes n}\xrightarrow{\cong} \bigotimes_C V$, the inclusion $i_\psi$ followed by the natural projection, is an isomorphism. Its inverse $\bigotimes _C  V \to V^{\otimes n}$ will be denoted as $w \mapsto (w)_\psi $.
\end{definition}

Here are some useful facts about the unordered tensor product, see \cite[Definition 10]{DJM} and \cite[Lemma 4]{M}.
\begin{lemma}\label{lemmaunordered}
\begin{enumerate}
\item[]
\item For an isomorphism $\psi\colon C \xrightarrow{\cong} [n]$ and a permutation $\sigma\colon [n] \to [n]$
$$ (w)_{\sigma\psi} = \sigma (w)_\psi, \quad  w\in \bigotimes_C V\,.$$

\item Any isomorphism $\rho \colon C \xrightarrow{\cong} D$ defines an isomorphism $\rho \colon \bigotimes_C V \to \bigotimes_D V$ by
\begin{equation} \label{eq:iso_action}
    (\rho x)_{\varphi} = (x)_{\varphi\rho}, \quad   x\in \bigotimes_D V, \, \varphi\colon D \xrightarrow{\cong} [\card{(D)}]  \,.
\end{equation}
    
\item There is a canonical isomorphism
    $(\otimes_C V) \otimes (\otimes_D V)\cong \otimes_{C\sqcup D} (V)$,
    given by ordering on $C\sqcup D$ induced
    on the orderings on $C$ and $D$, i.e. by
    $$ (\bigotimes_C V) \otimes (\bigotimes_D V) 
    \xrightarrow{(-)_\psi \otimes (-)_\phi }
    V^{\otimes ( \card{(C)} + \card{(D)} )} \xleftarrow{ (-)_{\psi\sqcup \phi} } \bigotimes_{C\sqcup D} V$$
    where $\psi\sqcup \phi$ is the induced ordering on $C\sqcup D$ from $\psi\colon C \xrightarrow{\cong} [\card{(C)}]$ and $\phi\colon D \xrightarrow{\cong} [\card{(D)}]$.

    The composition $(\otimes_C V) \otimes (\otimes_D V)\cong \otimes_{C\sqcup D} (V) \cong (\otimes_D V) \otimes (\otimes_C V)$ is the monoidal symmetry
    $\tau \colon (\otimes_C V) \otimes (\otimes_D V)\to  (\otimes_D V) \otimes (\otimes_C V)$.
\end{enumerate}
\end{lemma}
For $c\in C$, it makes sense to talk about the $c$-th element
of $\bigotimes_C V^*$; for example we can contract it with
$v\in V$. This operation is captured in the following definition.
\begin{definition}
For $v\in V$ and a finite set $C$ of cardinality $n$, let us define a \textbf{positional derivative}
$$\partial^{(c)}_v \colon \bigotimes_{C \sqcup \{c \}} V^* \to \bigotimes_{C} V^*$$
by setting, for  $f\in \bigotimes_{C \sqcup \{c \}} V^*$,
\begin{equation}
    (\partial^{(c)}_v f)_\psi=  v\otimes 1_{(V^*)^{\otimes n}} (f)_{\tilde{\psi}} 
\end{equation}
where, for arbitrary $\psi\colon C \xrightarrow{\cong} [n]$, the map $\tilde\psi\colon {C \sqcup \{c \}} \xrightarrow{\cong} [n+1]$ is defined by  $\tilde\psi(c) = 1$ and $\tilde\psi(c') = \psi(c') + 1$ for $c'\in C$. On the right hand side, we see 
$v\in V$ as a map $V^* \to \Bbbk$ via $\alpha \mapsto (-1)^{\hdeg{v}\hdeg{\alpha}}\alpha(v)$. 
\end{definition}

Here we collect some of the useful properties of the positional derivative.
\begin{lemma}\label{lemmaposder}
\begin{enumerate}
    \item Under the isomorphism
    $\bigotimes_{C\sqcup \{c\} \sqcup D} V^* \cong (\bigotimes_{C\sqcup \{c\}} V^*) \otimes (\bigotimes_{ D}V^*)$, the positional derivative $\partial^{(c)}_v$ is sent to $\partial^{(c)}_v \otimes 1_{\bigotimes_D V^*}$
    
    \item   For $c \in C$, and $\rho \colon C \xrightarrow{\cong} D$ we have $\rho|_{C\setminus \{c\}} \partial^{(c)}_v =  \partial^{(\rho(c))}_v \rho$.
    
    \item The positional derivatives graded-commute, i.e. 
    $\partial^{(c)}_v \partial^{(d)}_w = (-1)^{\hdeg v \hdeg w} \partial^{(d)}_w \partial^{(c)}_v$.
\end{enumerate}
\end{lemma}

Now, we can turn to the definition of an endomorphism modular operad.

\begin{definition}
Let $(V,d)$ be a dg vector space which is degree-wise finite-dimensional. An \textbf{odd symplectic form} $\omega\colon V\otimes V \rightarrow \Bbbk$ of degree $-1$ is a nondegenerate graded-antisymmetric bilinear map\footnote{Note, this means that $\omega(u,v)\neq 0$ implies $|u|+|v|=1$ and $\omega(v,u)=(-1)^{|v| |u|+1}\omega(u,v)=-\omega(u,v)$.}. If $d(\omega)=0$, i.e.
$$\omega\circ (d\otimes 1_V + 1_V\otimes d)=0,$$
we call $(V,d, \omega)$ a \textbf{dg symplectic vector space}.
\end{definition}

If $\lbrace e_l\rbrace$ is a homogeneous basis of $V$  and  $\omega_{kl} = \omega(e_k, e_l)$, we define
$$e^k= \sum_{l}^{} (-1)^{|e_l|}\omega^{kl} e_l\,. $$

Note that $\omega^{kl}$, defined by $\sum_l \omega^{kl} \omega_{lm} = \delta^k_m$, is well defined, for $V$ degree-wise finite-dimensional. This is because the infinite matrix $\omega_{ij}$ is nonzero only in finite blocks corresponding to $V_k$ and $V_{1-k}$, and we only need to invert those blocks. Similarly, the sum in the definition of $e^k$ is well defined, since it only has finite number of nonzero terms for fixed $k$.  The fact that $\omega$ is degree $-1$ implies $|e^k|=1-|e_k|$. 

\begin{definition}\label{def_endomorph}
The \textbf{odd endomorphism modular operad}  $\mathcal{E}_V$ is the odd modular operad defined by\footnote{Note that the tensor powers of $V^*$ are not degree-wise finite-dimensional, even for degree-wise finite-dimensional $V^*$.}
$$\mathcal{E}_V (C,G)= \bigotimes_C V^*\,,$$
 with an action of $\rho\colon C\rightarrow D$ given by \eqref{eq:iso_action}.

The compositions and self-compositions are defined as follows. If $f\in \mathcal{E}_V(C_1\sqcup\{a\},G_1)$ and $g\in\mathcal{E}_V(C_2\sqcup\{b\},G_2)$,
then 
$$(f\oOo{a}{b}g) =  \sum_k (-1)^{\hdeg f \hdeg{e^k}} (\partial^{(a)}_{e_k} f) \otimes (\partial^{(b)}_{e^k} g),$$
where we use the canonical isomorphism  $(\otimes_{C_1} V^*) \otimes (\otimes_{C_2} V^*)\cong \otimes_{C_1\sqcup C_2} (V^*)$.

The self-composition for $f\in \bigotimes_{C\sqcup \{a,b\}}V^*$ is given by 
\begin{equation} \label{eq:selfcompEndV} \oOo{}{ab}f =  \sum_k \partial^{(a)}_{e_k} \partial^{(b)}_{e^k} f.\end{equation}
This is well defined because $f$ is a finite sum of tensor products of elements of $V^*$.

This operad is equipped with a differential given for $f\in\bigotimes_C V^*$ by
$$(df)_\psi = d_{(V^*)^{\otimes \card C}} (f)_\psi\, .$$
where the differential on $\alpha \in V^*$ is $(d\alpha)(v) = (-1)^{\alpha +1} \alpha(dv)$
is defined so that the pairing $V^*\otimes V \to \Bbbk$ is a chain map.
\end{definition}
\begin{lemma}
This defines a structure of an odd modular operad on $\mathcal E_V$.
\end{lemma}
\begin{proof}
These operations agree with the \emph{standard definition} of an odd endomorphism modular operad of Markl, i.e.
\cite[Equations (12b) and (13b)]{M}, just that we use $\bigotimes_C (V^*)$ instead of $(\bigotimes_C V)^*$. This is because the partial derivative $ \partial^{(a)} \colon  V\otimes\bigotimes_{C\sqcup \{a\} }V^* \to \bigotimes_C V^*$ is the composition $V\otimes \bigotimes_{C\sqcup \{a\} }V^*  \cong V\otimes V^* \otimes \bigotimes_{C}V^* \xrightarrow{\mathrm{ev}\otimes 1}  \bigotimes_C V^*  $. We get the standard definition because our the degree 1 tensor $e_i\otimes e^i$ comes from left in the definition of $\oOo{a}{b}$ and $\oOo{}{ab}$.

Alternatively, the (odd versions of) axioms from Definition \ref{DEFModOp} follow easily from Lemma \ref{lemmaunordered} and Lemma \ref{lemmaposder}. 
For example, the axiom (MO3) is
$$ \oOo{a}{b} (f\otimes g) =  \partial^{(a)}_{e_k} \otimes \partial^{(b)}_{e^k} (f \otimes g) = (-1)^{\hdeg f \hdeg g}
\partial^{(a)}_{e_k} \otimes \partial^{(b)}_{e^k} \tau (g \otimes f) = (-1)^{\hdeg f \hdeg g}
( \tau \circ \partial^{(b)}_{e^k} \otimes \partial^{(a)}_{e_k})  (g \otimes f)$$
 and $\sum_k e_k\otimes e^k = \sum_{kl} (-1)^{\hdeg {e_l}} \omega^{kl} e_k \otimes e_l = \sum_{l} e^l \otimes e_l $,
 which holds in the direct product $\prod_i V_i \otimes V_{1-i}$. Thus, using Item 3 of Lemma \ref{lemmaunordered}, we get that the right hand side is indeed  $(-1)^{\hdeg f \hdeg g} \oOo{b}{a} (g\otimes f)$.
\end{proof}

Now, we can define the connected sum on the endomorphism operad.
\begin{definition}
Define $$\#_1 \colon \mathcal{E}_V(C,G) \to \mathcal{E}_V(C,G+2)$$ to be the identity on $\bigotimes_C V^*$ and define
$$\#_2 \colon \mathcal{E}_V(C_1,G_1) \otimes \mathcal{E}_V(C_2,G_2) \to \mathcal{E}_V(C_1\sqcup C_2,G_1+G_2 +1)$$ to be the canonical isomorphism
$(\otimes_{C_1} V^*) \otimes (\otimes_{C_2} V^*)\cong \otimes_{C_1\sqcup C_2} (V^*)$.
\end{definition}

\begin{lemma}\label{connected_sum_via_concatenation}
The odd modular operad $\mathcal{E}_V$, with the above defined operations $\#_1$ and $\#_2$, is an odd modular operad with a connected sum.
\end{lemma}
\begin{proof}
(CS1) follows easily from the definition of the action of isomorphisms \eqref{eq:iso_action}. (CS2) follows from item 3 of Lemma \ref{lemmaunordered}. (CS3) follows from associativity of the tensor product. (CS4) and (CS5b) are trivial, since $\#_1$ is the identity.  The remaining axioms follow from the Lemma \ref{lemmaposder}, for example (CS5a) gives
$$ \oOo{a}{b} \#_2= \sum_{k} \partial^{(a)}_{e_k} \partial^{(b)}_{e^k} \#_2 $$
and commuting the positional derivatives through $\#_2$, we get the four possibilities via Item 1 of Lemma \ref{lemmaposder}.
\end{proof}

\begin{remark}
To encode a modular operad $\oP$, it is enough to keep the spaces $\oP([n], G)$. The operations then involve a choice of ordering on e.g. $[n_1] \sqcup [n_2]$ for $\#_2$. Choosing the ordering by placing $[n_1]$ before $[n_2]$, the operadic structure map of $\mathcal E_V$ acquire a particularly simple form \cite[III.D, E]{DJM}; the connected sum $\#_2$ turns into the  identification
\[ \mathcal E([n_1], G_1) \otimes \mathcal E([n_2], G_2) = (V^*)^{\otimes n_1}\otimes (V^*)^{\otimes n_2} \xrightarrow[\#_2]{\cong} (V^*)^{\otimes (n_1 + n_2)} = \mathcal E([n_1+n_2], G_1+G_2 + 1)\,. \]
Since we replaced the category of corollas by its skeleton $\{ ([n], G)\}$, this version of a modular operad is usually called skeletal.
\end{remark}

\section{Beilinson-Drinfeld algebras and the connected sum}
 In \cite[Section~5]{BarannikovModopBV}, Barannikov introduced a dg Lie algebra structure on the (shifted) space of formal $\oP$-functions, for a modular operad $\oP$. If $\mathcal P$ is endowed with a connected sum, this space of functions acquires a commutative product and becomes a Beilinson-Drinfed algebra, a close relative of a Batalin-Vilkovisky algebra.

\subsection{Beilinson-Drinfeld algebras}\label{sec_connected_sum_ME}
 Bielinson-Drinfeld algebras, or BD algebras for short, appeared in the work of Beilinson and Drinfeld \cite{BeilinsonDrinfeld}, see also \cite{gwilliam-thesis, CostelloGwilliam}.
\begin{definition}\label{BV_algebra}
	A \textbf{BD algebra} is a graded commutative associative algebra on a graded module $\mathcal F$ over $\Bbbk [[\varkappa]]$, flat over $\Bbbk[[\varkappa]]$, with a bracket $\{, \}\colon \mathcal F^{\otimes 2}\to \mathcal F$ of degree 1 that satisfies
	\begin{align}
	\{X, Y\} &= -(-1)^{(\hdeg X +1) (\hdeg Y +1)} \{Y, X\} \, ,\nonumber\\
	\label{eq:Jacobi}
	\{ X, \{Y, Z\}\} &= \{ \{X, Y\}, Z\} + (-1)^{(\hdeg X+1)(\hdeg Y +1) } \{Y, \{X, Z\}\}\,,  \\
	\{X, YZ\} &= \{X, Y\}Z + (-1)^{(\hdeg X +1)\hdeg Y} Y \{X ,Z  \}\, , \nonumber
	\end{align}
	and a square zero operator $\Delta \colon \mathcal F \to \mathcal F$ of degree 1 such that 
	\begin{equation}\label{eq:BVgood}
	\Delta (XY) = (\Delta X) Y + (-1)^{\hdeg{X}} X \Delta Y + (-1)^{\hdeg X} \varkappa\{X, Y \}\, .
	\end{equation}
	If $\mathcal F$ is also equipped with a differential, we require $\Delta$ and the bracket to commute with it. For algebras with unit $1$, we will require $\Delta(1) = 0$. 
\end{definition}

\subsection{Formal functions associated to a modular operad}\label{ssec:formalfunctions}
Let us consider a modular operad $\oP$  and an odd modular operad $\oQ$. Define

\begin{align*}
    \Fun(\oP,\oQ)(n,G)  &=\left(\oP(n,G) \otimes \oQ(n,G)\right)^{\Sigma_n},\\
    \Fun(\oP,\oQ)       &= \prod_{n\geq 0}\prod_{G \geq 0}  \Fun(\oP,\oQ)(n,G).
\end{align*}
In \cite{BarannikovModopBV}, Barannikov introduced the following operations of degree 1, defined on components
\begin{eqnarray*}
 d&\colon& \Fun(\oP,\oQ)(n,G) \rightarrow \Fun(\oP, \oQ)(n,G), \\
    \Delta&\colon&\Fun(\oP,\oQ)(n+2,G) \rightarrow \Fun(\oP, \oQ)(n,G+1), \\
   \{-,-\}&\colon& \Fun(\oP,\oQ)(n_1+1,G_1)\otimes \Fun(\oP,\oQ)(n_2+1,G_2) \rightarrow \Fun(\oP, \oQ)(n_1+n_2,G_1+G_2),
\end{eqnarray*}
by
\begin{eqnarray}
d=d_{\oP}\otimes 1 - 1 \otimes d_{\oQ}, \nonumber \\
\Delta = (\ooo{}{ab}\otimes \oOo{}{ab})(\theta\otimes \theta) \label{BVlaplace_jako_selfcomposition},
\end{eqnarray}
for arbitrary bijection\footnote{ We write $\theta\otimes \theta$ instead of $\oP(\theta)\otimes \oQ(\theta)$ for brevity.}  $\theta\colon  [n+2] \xrightarrow{\cong} [n]\sqcup\{a,b\}$; and
\begin{equation}\label{definice_zavorky}
   \{X,Y\} = (-1)^{|X|}\cdot 2\sum_{C_1,C_2}(\ooo{a}{b}\otimes \oOo{a}{b})(\theta_1\otimes \theta_2 \otimes \theta_1\otimes \theta_2) (1\otimes \tau \otimes 1)(X\otimes Y),
\end{equation}
where we sum over all disjoint decompositions $C_1\sqcup C_2 = [n_1+n_2]$, such that $\card(C_1)=n_1$, $\card(C_2)=n_2$, the bijections\footnote{No summation over those.}  $\theta_1\colon [n_1+1]  \xrightarrow{\cong} C_1\sqcup \{ a\}$, $\theta_2\colon [n_2+1] \xrightarrow{\cong} C_2\sqcup \{b\}$ are chosen arbitrarily, and $\tau$ is the monoidal symmetry. These operations are then extended to the whole $\Fun(\oP, \oQ)$.

\begin{theorem}[{\cite{BarannikovModopBV}}]
The maps $d$, $\Delta$ and $\{\,,\,\}$ are well defined, independent of the choice of $\theta$, $\theta_1$, $\theta_2$ and they satisfy the following properties
\begin{eqnarray}
d^2 &=& 0, \nonumber \\
d \{,\} + \{,\}(d\otimes 1 + 1 \otimes d)&=&0, \nonumber \\
\Delta^2 &=& 0, \nonumber \\
\Delta \{,\} +\{,\}(\Delta\otimes 1+1\otimes \Delta)&=&0, \nonumber \\
\Delta d + d \Delta &=& 0, \nonumber
\end{eqnarray}
and the bracket satisfies
\begin{eqnarray*}
\{ X, Y\} &=&- (-1)^{(\hdeg{X}+1)(\hdeg{Y}+1)}\{Y,X\},\\
\{ X, \{Y, Z\}\} &=& \{ \{X, Y\}, Z\} + (-1)^{(\hdeg X+1)(\hdeg Y +1) } \{Y, \{X, Z\}\}.
\end{eqnarray*}
\end{theorem}
See \cite[Section~5]{BarannikovModopBV} and also \cite[Theorem~20]{DJM} for a more detailed proof (our bracket $\{X, Y\}$ equals $(-1)^{|X|}2$ times their bracket). This structure was called a \emph{generalized Batalin-Vilkovisky algebra} in \cite{DJM}, since it lacks a compatible commutative product.
\medskip

The motivation for this structure comes from the fact that morphisms from the  \emph{Feynman transform of $\oP$} to $\oQ$ are in bijection with degree-0 elements $S\in \Fun(\oP, \oQ)$ that satisfy the quantum master equation $dS + \Delta S + \frac 12 \{S, S\} = 0$, see \cite{BarannikovModopBV, DJM}. 

\subsection{Connected sum and a commutative product}
Now we introduce a commutative product and an operation $\sharp \colon \Fun(\oP,\oQ) \to \Fun(\oP,\oQ)$, coming from $\#_2$ and $\#_1$.
\begin{definition}\label{product_on_con_space} Let $\oP$ be a modular operad and $\oQ$ an odd modular operad. Moreover, assume each of them equipped with a connected sum. 
Define a \textbf{product} $$\star\colon \Fun(\oP,\oQ)(n_1,G_1) \otimes \Fun(\oP,\oQ)(n_2,G_2)\rightarrow \Fun(\oP,\oQ)(n_1+n_2,G_1+G_2+1)$$ as
\begin{equation}\label{connected_sum_on_CON}
    \star = \sum_{C_1,C_2}(\#_2 \otimes \#_2 )(\theta_1\otimes \theta_2\otimes \theta_1\otimes \theta_2)(1\otimes \tau\otimes 1),
    \end{equation}
where, as before, the sum runs over all disjoint decompositions $C_1 \sqcup C_2 = [n_1+n_2]$, such that $\card(C_1)=n_1$, $\card(C_2)=n_2$, the bijections  $\theta_1\colon [n_1] \xrightarrow{\cong} C_1$, $\theta_2\colon [n_2] \xrightarrow{\cong} C_2$ are chosen arbitrarily, and $\tau$ is the monoidal symmetry.

We also define the \textbf{operator $\sharp\colon \Fun(\oP,\oQ)(n,G)\rightarrow \Fun(\oP,\oQ)(n,G+2)$} as
\begin{equation}\label{hbar_operator}    \sharp = \#_1 \otimes \#_1\, . \end{equation}
Finally, we extend $\star$ and $\sharp$ on $\Fun(\oP, \oQ)$ linearly.
\end{definition}
\begin{lemma}
The maps $\star$, $\sharp$ are well defined  and $\star$ doesn't depend on the choice of $\theta_1, \theta_2$.
\begin{proof}
The product $\star$ is well defined since only finite number of terms contribute to the component $\Fun(\oP, \oQ)(n, G)$ of the result, i.e. those components $(n_1, G_1)$ and $(n_2, G_2)$ such that $n= n_1 + n_2$ and $G = G_1 + G_2 + 1$. The result is independent of choices of $\theta_i$, because different choices of $\theta_i$ differ by precomposition with a permutation of $[n_i]$, under which $\Fun(\oP,\oQ)(n_i,G_i)$ is invariant.
\end{proof}
\end{lemma}

\begin{theorem}\label{thm_BDproperties}
 If $\oP$ and $\oQ$ are as in Definition \ref{product_on_con_space}, then $\Fun(\oP, \oQ)$, with operations $d, \Delta, \{-,-\}$, $\sharp$ and $\star$ defined above, satisfies
\begin{enumerate}[label=\arabic*.]
    \item $\star$ is a commutative associative product, i.e. on elements:
    \begin{equation}\label{vlastnosti_nasobeni}
        X\star Y = (-1)^{|X|\cdot |Y|}Y\star X \qquad \mathrm{and} \qquad (X\star Y)\star Z= X\star (Y \star Z)\, .
    \end{equation}

    \item $\Delta \star =\star (\Delta \otimes 1) + \star (1\otimes \Delta)+ \sharp \{-,-\}$,
     i.e. on elements:
    \begin{equation}\label{kompatibilita_delty_nasobeni}
        \Delta(X\star Y)=(\Delta X)\star Y+ (-1)^{|X|} X\star (\Delta Y) + (-1)^{|X|} \sharp \{X,Y\}\, .
    \end{equation}
    \item $\{-,-\}(1\otimes \star)= \star (\{-,-,\}\otimes 1)+\star(1\otimes \{-,-,\})(\tau\otimes 1)$, i.e. on elements:
    \begin{equation}\label{kompatibilita_zavorky_nasobeni}
        \{X,Y\star Z\}=\{X,Y\}\star Z + (-1)^{|X|\cdot |Y| +|Y|} Y\star \{X,Z\}\, .
    \end{equation}
    
    \item The maps $\sharp$ and $\star$ are chain maps with respect to the differential $d$.
    
    \item  The map $\sharp$ commutes with the other operations, i.e. $\Delta \sharp = \sharp\Delta$,  $\, \{-, -\} (1\otimes \sharp) = \{-, -\} (\sharp\otimes 1)  = \sharp \{-, -\}$ and $\star (1\otimes \sharp) = \star (\sharp\otimes 1) = \sharp\star$. On elements, this gives
    \begin{align}
        \Delta (\sharp X) &= \sharp (\Delta X), \\
        \{X, \sharp Y\} = \{\sharp X, Y\} &= \sharp \{ X, Y\}, \\
        X\star(\sharp Y) =(\sharp X) \star Y & = \sharp (X\star Y). \label{eq:starsharpcomp}
    \end{align}
    
\end{enumerate}
\begin{proof}
Let $X=\sum_i x^i_{\oP}\otimes x^i_{\oQ} \in \Fun(\oP, \oQ)(n_X,G_X)$, where $x^i_{\oP}\in \oP(n_X,G_X)$ and $x^i_{\oQ}\in \oQ(n_X,G_X)$. For sake of brevity, we will omit the summation over $i$ (including the index) from the notation. Hence, we will write $X=x_{\oP}\otimes x_{\oQ}$. Similarly, we write $Y=\sum_i y^i_{\oP}\otimes y^i_{\oQ}= y_{\oP}\otimes y_{\oQ}$ and $Z=\sum_i z^i_{\oP}\otimes z^i_{\oQ}= z_{\oP}\otimes z_{\oQ}$ where $y^i_{\oP}\in \oP(n_Y,G_Y)$ etc. 
\medskip

\textbf{ The point 1.} follows from (CS1), (CS2) and (CS3). For commutativity:

\begin{eqnarray*}
\mkern-20muX\star Y \mkern-20mu &=& \mkern-18mu \sum_{C_1,C_2} (-1)^{|x_{\oQ}|\cdot |y_{\oP}|} (\theta_1 x_{\oP}\#_2\,  \theta_2 y_{\oP})\otimes (\theta_1 x_{\oQ}\#_2\,  \theta_2 y_{\oQ}) \\
\mkern-20muY\star X \mkern-20mu &=&\mkern-18mu \sum_{C_1,C_2} (-1)^{|x_{\oP}|\cdot |y_{\oQ}|} (\theta_1 y_{\oP}\#_2\,  \theta_2 x_{\oP})\otimes (\theta_1 y_{\oQ}\#_2\,  \theta_2 x_{\oQ}) =\\
\mkern-20mu &\stackrel[]{\mathrm{(CS2)}}{\mathrm{=}}&\mkern-18mu \sum_{C_1,C_2} (-1)^{|x_{\oP}| |y_{\oQ}|+|x_{\oP}||y_{\oP}|+|x_{\oQ}||y_{\oQ}|} (\theta_2 x_{\oP}\#_2\,  \theta_1 y_{\oP})\otimes (\theta_2 x_{\oQ}\#_2\,  \theta_1 y_{\oQ})=(-1)^{|X|\cdot |Y|} X\star Y .
\end{eqnarray*}

For associativity:
\begin{eqnarray*}
(X\star Y)\star Z= 
\sum_{C_3,C_4} (-1)^{|x_{\oQ}|\cdot |y_{\oP}|} \left( (\theta_3 x_{\oP}\#_2\,  \theta_4 y_{\oP})\otimes (\theta_3 x_{\oQ}\#_2\,  \theta_4 y_{\oQ})\right)\star Z= \\
=\sum_{\substack{C_1,C_2,\\C_3,C_4}} (-1)^{(|x_{\oQ}|+|y_{\oQ}|)\cdot |z_{\oP}|} (-1)^{|x_{\oQ}|\cdot|y_{\oP}|}(\theta_1 (\theta_3 x_{\oP} \#_2\,  \theta_4 y_{\oP})\#_2\,  \theta_2 z_{\oP})\otimes (\theta_1(\theta_3 x_{\oQ} \#_2\,  \theta_4 y_{\oQ})\#_2\,  \theta_2 z_{\oQ}),
\end{eqnarray*}
where $C_1\sqcup C_2= [n_x+n_y+n_z]$ and $C_3\sqcup C_4 = [n_x+n_y]$,  $\theta_1\colon [n_x+n_y] \xrightarrow{\cong} C_1$, $\theta_2\colon [n_z] \xrightarrow{\cong} C_2, \theta_3\colon [n_x]\xrightarrow{\cong} C_3, \theta_4\colon[n_y]\xrightarrow{\cong} C_4$ are chosen arbitrarily. From (CS1), 
we get $$\theta_1 (\theta_3 x_{\oP} \#_2\,  \theta_4 y_{\oP})= \theta_1(\theta_3 \sqcup \theta_4)(x_{\oP} \#_2\,  y_{\oP}),$$
where $(\theta_3 \sqcup \theta_4)\colon [n_x]\sqcup(n_x+[n_y])\xrightarrow{\cong} C_3\sqcup C_4=[n_x+n_y] $ and similarly for the $\oQ$-part.  
Therefore, we can rewrite the sums over decompositions $C_1\sqcup C_2$ and $C_3\sqcup C_4$ and actions of $\theta$'s as
\begin{gather*}
\sum_{E_1\sqcup E_2\sqcup E_3}  (-1)^{A} (\psi_1\sqcup \psi_2 \sqcup \psi_3) ((x_{\oP} \#_2\,  y_{\oP})\#_2\,  z_{\oP})\otimes (\psi_1\sqcup \psi_2 \sqcup \psi_3) ((x_{\oQ} \#_2\,  y_{\oQ})\# _2\, z_{\oQ}),
\end{gather*}
where $A= (|x_{\oQ}|+|y_{\oQ}|)\cdot |z_{\oP}|+|x_{\oQ}|\cdot|y_{\oP}|$, $\psi_1\colon [n_x]\xrightarrow{\cong} E_1$, $\psi_2\colon [n_y]\xrightarrow{\cong} E_2$, $\psi_3\colon [n_z] \xrightarrow{\cong} E_3$ and the sum is over all decompositions $E_1\sqcup E_2\sqcup E_3 = [n_x+n_y+n_z]$.  
Similarly, one gets
\begin{eqnarray*}
X\star (Y \star Z) =
\sum_{D_3,D_4}(-1)^{|y_{\oQ}|\cdot|z_{\oP}|} \left( X\star (\phi_3 y_{\oP}\#_2\,  \phi_4 z_{\oP})\otimes (\phi_3 y_{\oQ}\#_2\,  \phi_4 z_{\oQ})\right)
=\\
=\sum_{\substack{D_1,D_2,\\D_3,D_4}} (-1)^{|x_{\oQ}|\cdot ( |y_{\oP}|+|z_{\oP}|)} (-1)^{|y_{\oQ}|\cdot|z_{\oP}|} (\phi_1 x_{\oP}\#_2\,  \phi_2 (\phi_3 y_{\oP} \#_2\,  \phi_4 z_{\oP})\otimes (\phi_1 x_{\oQ}\#_2\,  \phi_2 (\phi_3 y_{\oQ} \#_2\,  \phi_4 z_{\oQ})),
\end{eqnarray*}
where $\phi_1\colon [n_x] \xrightarrow{\cong} D_1$, $\phi_2\colon [n_y+n_z]\xrightarrow{\cong} D_2$, $\phi_3\colon  [n_y]\xrightarrow{\cong}D_3$, $\phi_4\colon [n_z]\xrightarrow{\cong} D_4$.  Rewriting this as a sum over decompositions $E_1\sqcup E_2\sqcup E_3 = [n_x+n_y+n_z]$, this gives
\begin{gather*}
\sum_{E_1\sqcup E_2\sqcup E_3} \!\! (-1)^{A} (\psi_1\sqcup \psi_2 \sqcup \psi_3) (x_{\oP} \#_2\, ( y_{\oP}\#_2\, z_{\oP})) \otimes (\psi_1\sqcup \psi_2 \sqcup \psi_3) (x_{\oQ} \#_2\, (y_{\oQ}\#_2\, z_{\oQ}))\, .
\end{gather*}
By (CS3), we finally get $(X\star Y)\star Z= X\star (Y \star Z)$.
\\
\medskip

\textbf{The point (2.} follows from (CS5a). The left side of the required equality is:
\begin{gather*}
    \Delta(X\star Y) = \sum_{C_1, C_2}  \ooo{}{ab} \phi (\theta_1 x_{\oP} \#_2\, \theta_2 y_{\oP})\otimes \oOo{}{ab} \phi (\theta_1 x_{\oQ} \#_2\, \theta_2 y_{\oQ}) (-1)^{|x_{\oQ}||y_{\oP}|+|x_{\oP}|+|y_{\oP}|},
\end{gather*}
where $C_1\sqcup C_2 = [n_x+n_y]$ and where we have chosen $\phi=1_{[n_x+n_y]}$ (i.e. $a=n_x+n_y-1, b=n_x+n_y$). Now, we split the sum by distinguishing four cases according to positions of $a,b$ in the decomposition $C_1\sqcup C_2$ (cf. axiom (CS5a)). See also Figure \ref{fig:BDaxiom}. 
\begin{eqnarray*}
\Delta(X\star Y) &= &\sum_{\substack{C_1, C_2 \\ a,b\in C_1}}   (\ooo{}{ab}\theta_1 x_{\oP})\#_2\, \theta_2 y_{\oP} \otimes (\oOo{}{ab}\theta_1 x_{\oQ})\#_2\, \theta_2 y_{\oQ} \ (-1)^{|x_{\oQ}||y_{\oP}|+|x_{\oP}|+|y_{\oP}|}+ \\
 &+& \sum_{\substack{C_1, C_2 \\ a,b\in C_2}}   \theta_1 x_{\oP}\#_2\, (\ooo{}{ab}\theta_2 y_{\oP}) \otimes \theta_1 x_{\oQ}\#_2\, (\oOo{}{ab}\theta_2 y_{\oQ}) \ (-1)^{|x_{\oQ}||y_{\oP}|+|x_{\oP}|+|y_{\oP}|+|x_{\oQ}|} + \\
 &+& \sum_{\substack{C_1, C_2 \\ a\in C_1,b\in C_2}} \#_1\, \left(  \theta_1 x_{\oP}\ooo{a}{b} \theta_2 y_{\oP} \right) \otimes \#_1\, \left( \theta_1 x_{\oQ}\oOo{a}{b} \theta_2 y_{\oQ}\right) \ (-1)^{|x_{\oQ}||y_{\oP}|+|x_{\oP}|+|y_{\oP}|} + \\
 &+& \sum_{\substack{C_1, C_2 \\ a\in C_2,b\in C_1}}   \#_1\, \left(\theta_1 x_{\oP}\ooo{b}{a} \theta_2 y_{\oP} \right)\otimes\#_1\, \left( \theta_1 x_{\oQ}\oOo{b}{a} \theta_2 y_{\oQ}\right) \ (-1)^{|x_{\oQ}||y_{\oP}|+|x_{\oP}|+|y_{\oP}|} \, .
\end{eqnarray*}
It is easy to verify that the third and fourth lines give the same result. We compare the previous calculation with
\begin{gather*}
    (\Delta X)\star Y = \sum_{C_1,C_2} (\theta_1 \ooo{}{ab} \phi x_{\oP}) \#_2\, \theta_2 y_{\oP} \otimes (\theta_1 \oOo{}{ab} \phi x_{\oQ}) \#_2\, \theta_2 y_{\oQ}\ (-1)^{|x_{\oP}|+(1+|x_{\oQ}|)|y_{\oP}|},
\end{gather*}
where $C_1\sqcup C_2 = [n_x+n_y-2]$ with the choice $\phi=1_{[n_x]}$ and $a=n_x-1, b=n_x$,
\begin{gather*}
    (-1)^{|X|} X\star (\Delta Y)= \sum_{C_1, C_2} \theta_1 x_{\oP} \#_2\, \theta_2 (\ooo{}{ab} \phi y_{\oP}) \otimes \theta_1 x_{\oQ} \#_2\, \theta_2 (\oOo{}{ab} \phi y_{\oQ}) \ (-1)^{|x_{\oP}|+ |x_{\oQ}| + |y_{\oP}|+|x_{\oQ}||y_{\oP}|},
\end{gather*}
where we take $\phi=1_{[n_y]}$ and $a=n_y-1, b=n_y$, and 
\begin{gather*}
   (-1)^{|X|} \sharp \{X,Y\} = 2 \sum_{C_1,C_2}  \#_1\,\left(\theta_1 x_{\oP}\ooo{a}{b} \theta_2 y_{\oP}\right) \otimes \#_1\,\left(\theta_1 x_{\oQ}\oOo{a}{b} \theta_2 y_{\oQ}\right) \ (-1)^{|x_{\oQ}||y_{\oP}|+|x_{\oP}|+|y_{\oP}|}\, .
\end{gather*}
It is now easy to see that the required equality holds.
\medskip

\begin{figure}[h]
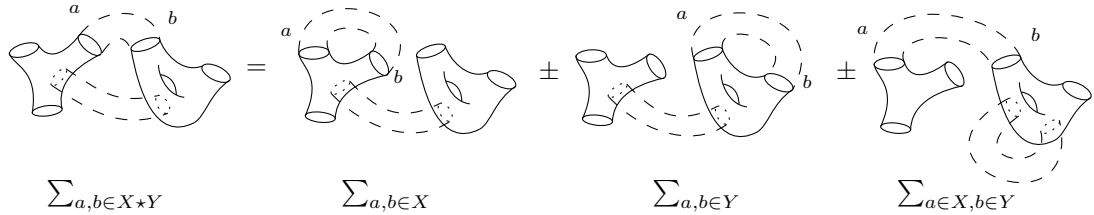

\begin{center}
\PICBDalg
\caption{Equation \eqref{kompatibilita_delty_nasobeni} pictorially. On the LHS, the operator $\Delta$ acts on a connected sum of two surfaces, connecting all pairs of punctures. On the RHS, we see three possible cases, depending on whether the punctures are both on the first surface, the second surface or there is one puncture on each surface. In the last case, the result has additional handle, giving the term $\sharp \{X,Y\}$ of \eqref{kompatibilita_delty_nasobeni}.}
\label{fig:BDaxiom}
\end{center}
\end{figure}

\textbf{The point 3.} follows from (CS1) and (CS6). First observe that:
\begin{gather*}
    \{X, Y\star Z\} = 2\sum_{\substack{C_1,C_2\\ D_1,D_2}}\left(\phi_1 x_{\oP} \ooo{a}{b} \phi_2(\theta_1 y_{\oP} \#_2\, \theta_2 z_{\oP})\right) \otimes \left(\phi_1 x_{\oQ} \oOo{a}{b} \phi_2(\theta_1 y_{\oQ} \#_2\, \theta_2 z_{\oQ})\right) \ (-1)^{B},
\end{gather*}
where we sum over all  decompositions $C_1\sqcup C_2 = [n_y+n_z], D_1\sqcup D_2 = [n_x+n_y+n_z-2]$ and $\theta_1\colon [n_y]\xrightarrow{\cong}C_1$, $\theta_2\colon [n_z]\xrightarrow{\cong}C_2$, $\phi_1\colon [n_x]\xrightarrow{\cong}D_1\sqcup\{a\}$, $\phi_2\colon [n_y+n_z]\xrightarrow{\cong}D_2\sqcup\{b\}$ are arbitrary bijections and $B=|y_{\oQ}|\cdot|z_{\oP}| + |x_{\oQ}|\cdot(|y_{\oP}|+|z_{\oP}|)+ |x_{\oP}| + |y_{\oP}|+|z_{\oP}|+|X| $. We split the sum into two according to position of $b$ ($b\in \phi_2(C_1)$ or $b\in \phi_2(C_2)$) and compare  with the two terms on the right hand side of \eqref{kompatibilita_zavorky_nasobeni}. The first term is 
\begin{gather*}
    \{X,Y\}\star Z =2 \sum \left(\theta_1(\phi_1 x_{\oP} \ooo{a}{b} \phi_2 y_{\oP}) \#_2\, \theta_2 z_{\oP} \right) \otimes \left(\theta_1(\phi_1 x_{\oQ} \oOo{a}{b} \phi_2 y_{\oQ}) \#_2\, \theta_2 z_{\oQ} \right)\ (-1)^{C},
\end{gather*}
where we sum over all decompositions $C_1\sqcup C_2 = [n_x+n_y+n_z-2]$, $D_1\sqcup D_2 = [n_x+n_y]$ and $\phi_1\colon [n_x]\xrightarrow{\cong} D_1$, $\phi_2\colon [n_y]\xrightarrow{\cong} D_2$, $\theta_1\colon  [n_x+n_y-2] \xrightarrow{\cong} C_1$, $\theta_2\colon [n_z]\xrightarrow{\cong}C_2$ are arbitrary bijections, $C=|x_{\oQ}|\cdot |y_{\oP}|+ |x_{\oP}| +|y_{\oP}|+ |z_{\oP}|\cdot(|x_{\oQ}|+|y_{\oQ}|+1)+|X|$ and  $a \in D_1$, $b\in D_2$ are arbitrary.  The second term is
\begin{gather*}
    Y\star \{X,Z\} = 2\sum \left(\theta_1 y_{\oP} \#_2\, \theta_2 (\phi_1 x_{\oP} \ooo{a}{b} \phi_2 z_{\oP})\right) \otimes  \left(\theta_1 y_{\oQ} \#_2\, \theta_2 (\phi_1 x_{\oQ} \oOo{a}{b} \phi_2 z_{\oQ})\right) (-1)^{D},
\end{gather*}
where we sum over all decompositions $C_1\sqcup C_2 = [n_x+n_y+n_z-2]$, $D_1\sqcup D_2 = [n_x+n_z]$ and $\phi_1\colon [n_x]\xrightarrow{\cong} D_1$, $\phi_2\colon [n_z]\xrightarrow{\cong} D_2$, $\theta_1\colon [n_y] \xrightarrow{\cong}C_1$, $\theta_2\colon [n_x+n_z-2]\xrightarrow{\cong}C_2$ are arbitrary bijections, $D= 
|x_{\oQ}|\cdot |z_{\oP}| + |x_{\oP}|+ |z_{\oP}|+|y_{\oQ}|\cdot(|x_{\oP}|+|z_{\oP}|)+|X|$  and $a\in D_1$, $b\in D_2$ are arbitrary.

Using (CS1) and (CS6) and collecting all the signs, one gets $$\{X,Y\star Z\}=\{X,Y\}\star Z + (-1)^{|X|\cdot |Y| +|Y|} Y\star \{X,Z\}\, .$$

\textbf{The point 4.} follows directly from the definition of $\sharp$ and $\star$. In \textbf{point 5.}, the compatibility of $\sharp$ with $\Delta$ follows from (CS4). The equation $ \{ \sharp X, Y \} = \sharp \{X, Y \}$ follows directly from (CS5b), the remaining equality follows from the symmetry of the bracket $\{-,-\}$. Similarly, the compatibility of $\sharp$ and $\star$ follows from (CS3) and the symmetry of $\star$.
\end{proof}
\end{theorem}

\subsection{Beilinson-Drinfeld algebras}
Using the action of $\sharp$, we can turn $\Fun(\oP, \oQ)$ to a (non-unital) BD algebra.
\begin{definition}
    For $f = \sum_{i\ge 0} f_i \varkappa^i \in \fld [[\varkappa]]$ and $p = \sum_{n, G \ge 0} p_{n, G} \in \Fun(\oP, \oQ)$,
    define the action of $f$ on $p$ by
    \[ fp = \sum_{n, G, i \ge 0} f_i \, \sharp^i (p_{n, G}). \]

\end{definition}
    Note that only terms coming from $p_{n, G'}$ for $G'\leq G$ contribute to the component $\Fun(\oP,\oQ)(n,G)=\left(\oP(n,G) \otimes \oQ(n,G)\right)^{\Sigma_n}$, and thus the result is well-defined.
\begin{lemma}\label{lemma_flat}
    The space $\Fun(\oP, \oQ)$ equipped with the action of $\fld[[\varkappa]]$ becomes a graded module over $\fld[[\varkappa]]$ and the operations $d, \Delta, \{-,-\}$ and $\star$ are maps of graded modules. 

    This module is flat over $\fld[[\varkappa]]$ if and only if the maps $\sharp \colon (\oP(n, G)\otimes \oQ(n, G))^{\Sigma_n} \to (\oP(n, G+2)\otimes \oQ(n, G+2))^{\Sigma_n}$ are injective for all $n, G$.
\end{lemma}
Thus, if $\sharp$ is injective,  $\Fun(\oP, \oQ)$ becomes a non-unital BD algebra. Note that in all of our examples ($\mathcal{QC}$, $\mathcal{QO}$ and $\mathcal E_V$), all $\#_1$ are injective, which is a stronger condition than the injectivity of $\sharp$.

\begin{proof}
    The first part of the lemma follows from Theorem \ref{thm_BDproperties}.

    Since $\fld[[\varkappa]]$ is a PID, being flat is equivalent to being torsion-free, i.e. no non-zero element of $\Fun(\oP, \oQ)$ is annihilated by a non-zero element of $\fld[[\varkappa]]$ \cite[Corollary~6.3]{Eisenbud}. This is furthermore equivalent to $\mathrm{Ker}\varkappa = 0$,
    since any non-zero element of $\fld[[\varkappa]]$ is equal to $\varkappa^i$ up to an invertible element, and if $\varkappa^i X =0$ for minimal $i$, then $\varkappa^{i-1}X \neq 0$ is annihilated by $\varkappa$. Let us now show the two implications. 
    
    If $\sharp (X) = 0$ for some nonzero invariant $X \in \oP(n, G)\otimes \oQ(n, G)$, then $\varkappa\in \fld[[\varkappa]]]$ annihilates $X$. On the other hand, let us suppose that  $\varkappa$ annihilates an element $\sum x_{n, G}$. Then each of the summands, an element of $\Fun(\oP, \oQ)(n, G)$, is in the kernel of $\sharp$.    
\end{proof}

\subsection{Quantum master equation}
To be able to talk about the exponentials $e^{S/\varkappa}$, we need to introduce negative powers of $\varkappa$. To avoid various convergence issues, we will restrict the possible negative powers of $\varkappa$. See Remark \ref{remark:weight} at the end of this section explaining the motivation for the following definition.
\begin{definition}\label{def:weight}
Consider the space of fixed \emph{weight} $w \in \frac 12\mathbb Z$.
\[ \tilde{\mathrm F}_w \equiv \bigoplus_{n/2 + G + 2q +1 = w}\Fun(\oP, \oQ)(n, G)\otimes \fld \varkappa^q \]
where by $\fld \varkappa^q, q\in \mathbb Z$, we mean a 1-dimensional vector space, with a generator $\varkappa^q$. 

Similarly, let
\[ \tilde{\mathrm F}_{\ge w}  = \prod_{\tilde w \ge w}   \tilde{\mathrm{F}}_{\tilde w} \] 
\end{definition}
With the multiplication given by $\star$ and by $\varkappa^{q_1} \star \varkappa^{q_2} = \varkappa^{q_1+q_2}$, this space becomes a graded-commutative algebra, with operations $d$, $\Delta$ and
$\{-,-\}$ extended by $\varkappa$-linearity (the bracket is possibly defined only partially, since it decreases the weight by $2$). The action $\varkappa \colon \varkappa^q\mapsto \varkappa^{q+1}$ makes $\tilde{\mathrm F}_{\ge w}$ into a $\fld[[\varkappa]]$-module.
\begin{definition}\label{def:funexp}
Define the space $\Fun_\mathrm{Exp}(\oP, \oQ)$ by the following quotient
\[ \Fun_\mathrm{Exp}(\oP, \oQ) \equiv \tilde{\mathrm F}_{\ge \frac 12} \,  / \, \{ \sharp X - \varkappa X \mid X \in \tilde{\mathrm F}_{\ge -\frac 32} \}. \]
\end{definition}
\begin{lemma}\label{lemma:weight}\begin{enumerate}[label=(\arabic*)]
    \item The space $\Fun_\mathrm{Exp}(\oP, \oQ)$ inherits the algebra structure, action of $\fld[[\varkappa]]$ and the operations $\star$, $d$, $\Delta$ and $\{-,-\}$, with the bracket defined only for arguments of total weight $\ge 5/2$. In the inherited weight grading, the maps $\star$, $d$, $\Delta$ have weight $0$, the bracket has weight $-2$ and $\varkappa$ has weight $2$. As a $\fld[[\varkappa]]$-module, it is flat.
    
    \item The natural map $\iota \colon \Fun(\oP, \oQ) \to \Fun_\mathrm{Exp}(\oP, \oQ)$,  with the image in weight $> 2$ of $\Fun_\mathrm{Exp}(\oP, \oQ)$, is a map of BD-algebras. It is injective iff the condition from Lemma \ref{lemma_flat} is satisfied, i.e. if  the maps $\sharp \colon (\oP(n, G)\otimes \oQ(n, G))^{\Sigma_n} \to (\oP(n, G+2)\otimes \oQ(n, G+2))^{\Sigma_n}$ are injective for all $n, G$.
\end{enumerate}
\end{lemma}
\begin{proof}
\begin{enumerate}[label=(\arabic*)]
    \item    By Theorem \ref{thm_BDproperties}, the subspace $J \equiv \{ \sharp X - \varkappa X \mid X \in \tilde{\mathrm F}_{\ge -\frac 32} \}$ is an ideal preserved by the BD-algebra maps. The weight grading is preserved since both $\sharp$   and multiplication by $\varkappa$ increase weight by $2$.
    \medskip
    
    To show the flatness w.r.t. the action of $k[[\varkappa]]$, it is enough to show that multiplication of $\varkappa$ is injective. Let $X\in \Fun_\mathrm{Exp}(\oP, \oQ)$ be such that $\varkappa X = 0$, i.e.
    $\varkappa (X + J) \in J$, i.e. $\varkappa X = \varkappa Y - \sharp Y$ for some $Y\in \tilde{\mathrm F}_{w\ge 12}$. Then $X = \varkappa (Y/\varkappa)- \sharp (Y/\varkappa) \in J$.  
    
   \item The map $\iota$ is well defined, since only elements with $n/2 + G = w-1$ contribute to the weight $w$
   component of $\Fun_\mathrm{Exp}(\oP, \oQ)$. The image of $\iota$ has weight $>2$ by the stability condition \eqref{eq:stability}.
   \medskip
   
   To show the injectivity of $\iota$, consider an element $Y\in \Fun(\oP, \oQ)$ which gets sent to the ideal $J$, i.e. $Y = \sharp X - \varkappa X$ for some $X\in \tilde{\mathrm{F}}_{\ge -\frac 32}$. Since the ideal $J$ is compatible with the weight grading, we can assume that $Y$ and $X$ have a definite weight.  We expand $X$ in powers of $\varkappa$
    \[ X = \sum_{-m}^n X_q \varkappa^{q} \]
    where the sum is finite thanks to the direct sum in definition of $\tilde{\mathrm{F}}_w$. Since $Y= \sharp X - \varkappa X$,  we get an equality of Laurent polynomials in $\varkappa$.
    \[ \sharp X_{-m} \varkappa^{-m} + (X_{-m} + \sharp X_{-m+1})\varkappa^{-m +1} + \dots + (X_{n-1} +\sharp X_n)\varkappa^n + X_n\varkappa^{n+1} = Y \varkappa^0 \]
    because $Y\in \Fun(\oP, \oQ)$ has no powers of $\varkappa$ in itself. If $\sharp$ is injective, then we obtain from this equality that $X_0 = X_{-1} = 0$, which implies $Y = 0$. On the other hand, if $\sharp$
    is not injective, an element $K$ of its kernel satisfies $K = \sharp(K/\varkappa) - \varkappa (K/\varkappa)$, which lies in the ideal from Definition \ref{def:funexp},
\end{enumerate}
\end{proof}
This allows us to define formal exponentials and logarithms. As an image of the exponential, we will consider $\Fun^\mathrm{Grp}_\mathrm{Exp}(\oP, \oQ)$, a multiplicative abelian group of elements $1 + X$ of with $X\in \Fun_\mathrm{Exp}(\oP, \oQ)$, on which the BD-algebra operations can be defined in an obvious way.  
\begin{definition}
    Define two maps
    \[  \exp(X)\colon \Fun_\mathrm{Exp}(\oP, \oQ) \leftrightarrows \Fun^\mathrm{Grp}_\mathrm{Exp}(\oP, \oQ) \colon \log{X} \]
    by 
    \[ \exp(X) = 1 + X + X^2/2! + X^3/3! + \dots\]
    and 
    \[  \log{1+X} = X - X^2/2 + X^3/3 + \dots \]
\end{definition}
\begin{lemma}
    These two maps are well-defined, mutually inverse maps. The exponential behaves with respect to the $\Delta$ as
    \[ \Delta (e^X) =  (\Delta X + \frac 12 \varkappa \{X, X\}) e^X.\]
\end{lemma}
\begin{proof}
It is a simple consequence of equations \eqref{kompatibilita_delty_nasobeni} and \eqref{kompatibilita_zavorky_nasobeni} that 
$$\Delta X^n = n X^{n-1} \Delta X + \dfrac{ n(n-1)}{2} \sharp \{X,X\} X^{n-2}\, .$$
Thus, for a power series $f(X)= \sum_{n\geq 0} f_n X^n$, we have in the quotient of Definition \ref{def:funexp}
$$\Delta(f(X))=\sum_{n\geq 0} f_n \left(n X^{n-1} \Delta X + \dfrac{n(n-1)}{2} \varkappa \{X,X\} X^{n-2}\right) = f'(X)\Delta X+ \dfrac{1}{2}  f''(X) \varkappa \{X,X\}\, .$$
\end{proof}

Thus we arrive at another characterization of morphisms from the Feynman transform of $\oP$.
\begin{corollary}\label{cor:QMEexponential}
    Assume that the condition on $\sharp$ from Lemma \ref{lemma_flat} is satisfied. Then, a degree 0 element $S \in \Fun(\oP, \oQ)$ satisfies the quantum master equation
    \[ (d + \Delta ) S + \frac 12 \{S, S\} = 0 \]
    if and only if 
    \[ (d + \Delta) e^{\iota(S)/\varkappa} = 0\]
    holds in $\Fun_\mathrm{Exp}(\oP, \oQ)$.
\end{corollary}
\begin{proof}
Thanks to the stability condition \eqref{eq:stability}, $\iota(S)/\varkappa$ has positive weight, and we have 
\[ 0 = (d+\Delta) e^{\iota(S)/\varkappa} =  \frac{ \iota(d S + \Delta S + \frac 12 \{S, S \}) }{\varkappa} e^{\iota(S)/\varkappa}\]
which is equivalent to the quantum master equation for injective $\iota$.
\end{proof}
\begin{remark}\label{remark:weight}
The weight $w = n/2 + G + 2q + 1$ is a generalization of the weight $2(g+q)+n$ introduced by Braun and Maunder \cite[Def.~2.8]{BraunMaunder}; this choice is motivated by $\Delta, \star, \sharp/\varkappa$ having weight 0. This weight, and the stability condition \eqref{eq:stability}, make it possible to define the expression $\Delta e^{\iota(S)/\varkappa}$. See also \cite[Sec.~2.2]{DJP} for similar considerations for the $\mathcal{QC}$ case. 
\medskip

The power of $\varkappa$ should be thought of as the geometric genus $g$, motivated by the relation $\varkappa = \sharp$ in Definition \ref{def:funexp}. Zwiebach uses powers of $\hbar$ to count $G$ in the open-closed string theory context \cite[Eq.~(3.1),~(3.11)]{ZwiebachOC}, which is why we used the letter $\varkappa$ instead. \end{remark}

\subsection{Examples}
We will now describe the BD algebra structure coming from the two modular operads $\mathcal{QC}$ and $\mathcal{QO}$. Apart from the connected sum and the induced commutative product, these algebras were described in \cite{DJM}. Using the commutative product, we obtain a slightly simplified description, since $d$, $\Delta$ and the bracket are can be specified on generators of the algebra. 
\smallskip

Let us fix an odd symplectic vector space $V$ with a symplectic
form $\omega$ and a differential $d$. Let $e_i$ be a basis of $V$, which determines the dual basis $\phi^i$ of $V^*$ and the matrix $\omega_{ij} = \omega(e_i, e_j)$ with inverse $\omega^{ij}$. 

\subsubsection[The operad QC]{The operad $\mathcal{QC}$}
The space of formal functions on $V$, recalled in the following definition, is a BD algebra. We will now show that (up to a few non-stable elements), this BD algebra is isomorphic to $\Fun(\mathcal{QC}, \mathcal E_V)$.
\begin{definition}\label{def:SymBV}
On $\Fun_\mathrm{sym}(V) =  \prod_{n\ge 0}\mathrm{Sym}^n (V^*) \otimes \fld[[\varkappa]]$, define $d$ and $\Delta$ 
\begin{align} 
    d  & = (-1)^\hdeg{\phi^i}(\phi^i \circ d_V) \frac{\partial}{\partial \phi^i},  \nonumber\\
    \Delta &= (-1)^{\hdeg{\phi^i}} \omega^{ij} \frac {\partial^2}{\partial \phi^i \partial \phi^j}\,, \label{eq:Deltasym}.
\end{align}
\end{definition}
The space $(\Fun_\mathrm{sym}(V), d, \Delta)$ is a BV algebra, and thus  $(\Fun_\mathrm{sym}(V), d, \varkappa \Delta, \{-,-\})$ is a BD algebra, where $(-1)^\hdeg{X}\{X, Y\} = \Delta(XY) - \Delta(X) Y - (-1)^{\hdeg{X}} X\Delta(Y)$. For completeness, this gives
\[ \{X, Y\} = (-1)^{\hdeg{\phi^i} + \hdeg{X}(\hdeg{\phi^j}+1)}  \omega^{ij} \frac{\partial X}{\partial \phi^i} \frac{\partial Y}{\partial \phi^j}.\]
\smallskip

Recall from Section \ref{ssec:formalfunctions} that the space $\Fun(\mathcal{QC}, \mathcal E_V)$ is spanned by $\Sigma_n$-invariant tensors of the form $C_n^g \otimes w$, where $C_n^g$ is the generator of $\mathcal{QC}(n, 2g+n/2-1)$ and $w\in (V^*)^{\otimes n}$.
\begin{lemma}
The map $\Psi\colon \Fun(\mathcal{QC}, \mathcal E_V) \to \Fun_\mathrm{sym}(V)$, given by
\[ C_n^g\otimes w \to (n!)^{-1} [w]\varkappa^g \]
 is an injective map of BD algebras over $\fld[[\varkappa]]$, with the image given by the elements with $2g + n > 2$. The map $w \mapsto [w]$ is the projection $(V^*)^{\otimes n} \to \mathrm{Sym}^n(V^*)$ given by $\phi_1 \ot\dots \ot \phi_n \mapsto  \phi_1 \dots \phi_n$, the graded-commutative product of $\phi_i$.  
\end{lemma}
\begin{proof}
The space $\mathcal{QC}(n, G)$ is the trivial representation of the permutation group $\Sigma_n$, and thus $\Fun(\mathcal{QC}, \mathcal E_V)(n, G)$ is the subspace of invariants in $(V^*)^{\otimes n}$. This implies that $\Psi$ is an injection with the image specified by the stability condition $2(G-1) + n > 0 \iff 2g + n > 2$. 

Compatibility with the action of $\varkappa$ is immediate. To check the compatibility  of $\Psi$ with products, note that the terms of the sum in \eqref{connected_sum_on_CON} differ only by an action of $\Sigma_{n_1 + n_2}$, as follows from Lemma \ref{lemmaunordered}. Thus, after the projection by $[-]$, all the $\binom{n_1 + n_2}{n_1}$ terms give the same contribution. Concretely, calculating the product, we get
$$ \Psi(C_{n_1}^{g_1}\otimes w_1\star C_{n_2}^{g_2}\otimes w_2) = \binom{n_1 + n_2}{n_1}\frac{1}{(n_1+n_2)!} [w_1\otimes w_2] \varkappa^{g_1+g_2}$$
which indeed equals
$$\Psi(C_{n_1}^{g_1}\otimes w_1)\cdot \Psi(C_{n_2}^{g_2}\otimes w_2) = \frac{1}{n_1!n_2!}[w_1][w_2]\varkappa^{g_1+g_2}.$$
This is thanks to the normalization of $\Psi$ and to the property $[w_1 \otimes w_2]= [w_1][w_2]$.
\smallskip

 As $\Psi$ is compatible with products, it is enough to check $d$  on linear elements and $\Delta$ on quadratic elements $C_2^g\otimes (\phi^i\otimes\phi^j +(-1)^{\hdeg{\phi^i}\hdeg{\phi^j}}\phi^j\otimes\phi^i)$. This is because these maps are determined by their values on such elements, possibly after multiplying with a high-enough power of $\varkappa$ to fulfill the stability condition. We discuss only the case of $\Delta$, defined in \eqref{eq:selfcompEndV} and \eqref{BVlaplace_jako_selfcomposition}, which sends the above quadratic element to 
\[ C_0^{g+1}\otimes (-1)^{\hdeg{\phi^i} + \hdeg{\phi^j}}[ \phi^i(e_k)\phi^j(e^k) +(-1)^{\hdeg{\phi^i}\hdeg{\phi^j}}\phi^j(e_k)\phi^i(e^k)] 
= 2 C_0^{g+1}\otimes(-1)^{\hdeg{\phi^i}} \omega^{ij}, \]
which $\Psi$ sends to $(-1)^{\hdeg{\phi^i}}  2 \omega^{ij} \varkappa^{g+1}$. This agrees with the action of $\varkappa\Delta$ from \eqref{eq:Deltasym} on $\phi^i\phi^j \varkappa^{g}$.
\end{proof}

\subsubsection[The operad QO]{The operad $\mathcal{QO}$}
Let $V$ be as before. We will now define a BD algebra structure on symmetric powers of cyclic words with letters from $V^*$. 
Related BV structures appeared for example in the work of Cieliebak, Latschev and Fukaya \cite[Section~10.]{CFL} and Barannikov \cite[Section~1.2]{BarannikovNCMI}. 
\begin{definition}
     The space of cyclic words in $V^*$ of length $k$
    is the space of coinvariants under the action of $\mathbb Z_k$ by cyclic permutations
    $$ \mathrm{Cyc}_k(V^*) = ((V^*)^{\otimes k})_{\mathbb Z_k}, $$
    with elements denoted by $(\phi_1 \dots \phi_n) = (-1)^{\hdeg{\phi_1} (\hdeg{\phi_2}+\dots + \hdeg{\phi_n})} (\phi_2 \dots \phi_n \phi_1)$.
    Then, we define the following algebra
    $$\Fun_\mathrm{cyc}(V) := \prod_{n\ge 0} \mathrm{Sym}^n (\bigoplus_{k\ge 1}\mathrm{Cyc}_k(V^*) ) [[\varkappa, \xi]]\,.$$
    This algebra carries a natural BD structure continuous in $\varkappa$ and $\xi$. The Laplacian is defined by
    $$ \Delta (\phi^{i_1} \dots \phi^{i_n} ) = \sum_{k=0}^{n-2} \pm \omega^{i_1 i_{k+2}} (\phi^{i_2} \dots \phi^{i_{k+1}}) (\phi^{i_{k+3}} \dots \phi^{i_n} )+\text{cycl.}$$
    where the sign $\pm$ in the first term is equal to $(-1)^{\hdeg{\phi^{i_1}} + \hdeg{\phi^{i_{k+2}}}( \hdeg{\phi^{i_2}} + \dots + \hdeg{\phi^{i_k+1}})}$. In the terms $k=0$ and $k=n-2$, one of the cyclic words is empty as is replaced by $\xi$. The remainder denoted ``$+\text{cycl.}$'' contains the $n-1$ terms obtained by cyclically permuting the indices $i_1 , \dots ,i_n$ in the first term and by multiplying by the Koszul sign of this cyclic permutation.

    On products of cycles, $\Delta$ is extended to a BD operator as in \eqref{eq:BVgood}, using the bracket
    $$ \{ (\phi^{i_1} \dots \phi^{i_{n_1}} ), (\phi^{j_1} \dots \phi^{j_{n_2}} )\} =   \pm 2 \omega^{i_1 j_1} (\phi^{i_2} \dots \phi^{i_{n_1}} \phi^{j_2} \dots \phi^{j_{n_2}}) + \text{cycl.$\times$cycl.}$$
    where the sign $\pm$ in the first term is equal to $(-1)^{\hdeg{\phi^{i_1}} + ( \hdeg{\phi^{i_1}} + \dots + \hdeg{\phi^{i_{n_1}}} )(\hdeg{\phi^{j_1}}+1)}$. The term ``$+ \text{cycl.$\times$cycl.}$'' consists of $n_1n_2-1$ terms obtained from the first term by cyclic permutations among indices $i$ and $j$, multiplied by the appropriate sign. For $n_1 = n_2 = 1$, the empty cycle is replaced by $\xi$. 
    
    The induced differential is given as in Definition \ref{def:SymBV}.
\end{definition}
In contrast to $\Fun_\mathrm{sym}(V)$, this BD algebra cannot be induced from a BV algebra by replacing $\Delta_\textnormal{BV}$ with $\varkappa\Delta_\textnormal{BV}$. This can be seen on the level of the operad $\mathcal{QO}$: the self-composition $\circ_{ab}$ applied on a disk is an annulus, which cannot be written as an image of $\#_1$. 
\begin{example}
    To illustrate the above formulas, let us give a few simple examples the BD structure defined above.
    \begin{align*}
        \Delta (\phi^a \phi^b) &= 2(-1)^\hdeg{\phi^a} \omega^{ab}\xi , \\
        \Delta (\phi^a\phi^b\phi^c) &= 2 ((-1)^\hdeg{\phi^a} \omega^{ab}(\phi^c)+ \text{cycl.})\xi \\
         &= 2\left((-1)^\hdeg{\phi^a} \omega^{ab}(\phi^c)+ (-1)^{\hdeg{\phi^a} + \hdeg{\phi^b}} \omega^{bc}(\phi^a)+ (-1)^{\hdeg{\phi^b}\hdeg{\phi^c} + \hdeg{\phi^c}} \omega^{ca}(\phi^b)\right)\xi , \\
         \{(\phi^a), (\phi^b)\} &= 2 \omega^{ab}\xi.
    \end{align*}
\end{example}

We would like to show that $\Fun_\mathrm{cyc}(V)$ contains $\Fun(\mathcal{QO}, \mathcal E_V)$, with $\varkappa$ counting the geometrical genus of the element of $\mathcal{QO}$ and $\xi$ counting the empty punctures. 

In this case, the $\Sigma_n$ invariants in $\mathcal{QO}(n, G)\otimes (V^*)^{\otimes n}$ can be described as follows (see also \cite[V.C]{DJM}): the $\Sigma_n$-representation $\mathcal{QO}(n, G)$ comes from the set of all cycles on letters $1 \dots n$, with its $\Sigma_n$-action given by renumbering. Orbits of this set-theoretic action are completely specified by sequences\footnote{These are subject to the obvious conditions $\sum_i b_i = b$ and $\sum_i ib_i = n$.} $(b_0, b_1, \dots)\in \mathbb N^\mathbb{N}$, where $b_i$ is  the number of cycles of length $i$. Choose the following element in each orbit
\begin{equation}\label{eq:standardb} x_\mathbf{b} := \underbrace{()\dots ()}_{b_0\text{ times}} \underbrace{(1)\dots (b_1)}_{b_1\text{ times}} \underbrace{( (b_1+1) (b_1+2)) \dots }_{b_2\text{ times}}  \;\dots \end{equation}
For each such admissible $\mathbf{b}$ and
$w \in (V^*)^{\otimes n}$ invariant under the stabilizer of $x_\mathbf{b}$, we have an invariant element
\begin{equation}\label{eq:symnonstab}
\sum_{\sigma \in \Sigma_n/\mathrm{Stab}(\mathbf{b})} \sigma x_\mathbf{b} \otimes \sigma w,
\end{equation} 
The space of invariants $(\mathcal{QO}(n, G)\otimes (V^*)^{\otimes n})^{\Sigma_n}$ is spanned by such elements\footnote{If $G$ is a finite group, $X$ a finite $G$-set and $W$ a $G$-representation, then $(\fld X \otimes W)^G \cong \bigoplus_{O_i} W^{\mathrm{Stab}_{x_i}}$, where the sum is over all orbits of the $G$-action in $X$ and $x_i\in O_i$ is an arbitrarily chosen element of the orbit.}.

Define a map $\Theta\colon (\mathcal{QO}(n, G)\otimes (V^*)^{\otimes n})^{\Sigma_n} \to \Fun_{\mathrm{cyc}}(V)$ by 
\[ \Theta\colon  \sum_{\sigma \in \Sigma_n/\mathrm{Stab}(\mathbf{b})} \sigma x_\mathbf{b} \otimes \sigma w_\mathbf b \mapsto 
\frac{1} {\prod_{i\ge 1} i^{b_i} b_i !} [w_\mathbf{b}] \varkappa^{g}\xi^{b_0},  \]
where $w \mapsto [w]$ is the composition
\[ (V^*)^{\otimes n} \to \bigotimes_{i\ge 1} \left(\mathrm{Cyc}_{i}(V^*) \right)^{\otimes b_i} \to 
\bigotimes_{i\ge 1} \mathrm{Sym}^{b_i}\left(\mathrm{Cyc}_{i}(V^*) \right) \hookrightarrow
\Fun_{\mathrm{cyc}}(V)\,.\]

\begin{lemma}
    The map $\Theta \colon \Fun(\mathcal{QO}, \mathcal E_V) \to \Fun_\mathrm{cyc}(V)$ is an injective map of BD algebras over $\fld[[\varkappa]]$, with the image given by elements with $2g + b+ n/2> 2$. Here, $b$ is $b_0$ plus the total number of cyclic words, $n$ is the total number of letters.
\end{lemma}
\begin{proof}
The injectivity of $\Theta$ follows from the fact that the map $w \mapsto [w]$ is an isomorphism from invariants to coinvariants for the stabilizer subgroup of $\mathbf b$.
\smallskip

Let us check the compatibility of $\Theta$ with the products.  In the product of two elements
\[ \sum_{\sigma} \sigma x_{\mathbf{b}^{(1)}} \otimes \sigma w^{(1)} \star 
\sum_{\sigma'} \sigma' x_{\mathbf{b}^{(2)}} \otimes \sigma' w^{(2)}  \] 
there are $\prod_{i\ge 1} \binom{b^{(1)}_i + b^{(2)}_i}{b^{(1)}_i}$ contributions to the term $x_{\mathbf b^{(1)} + \mathbf b^{(2)}} \otimes W$, and their contributions to the tensor $W$ all differ by a permutation only among cycles of the same length, i.e. a permutation stabilizing $x_{\mathbf b^{(1)} + \mathbf b^{(2)}}$. Thus, in $[W]$, they are all equal, and the combinatiorial factor $\prod_{i\ge 1} \binom{b^{(1)}_i + b^{(2)}_i}{b^{(1)}_i}$ cancels thanks to $\prod 1/{b_i !}$ in the definition of $\Theta$. 
\smallskip

Using the compatibility of $\Theta$ with products, it is now enough to check $d$, $\Delta$ and the bracket on elements with only one cycle.  

Let us calculate $\Delta (\phi^{i_1} \dots \phi^{i_n})$. The cyclic word $(\phi^{i_1} \dots \phi^{i_n})$ can by obtained by $\Theta$ from the element 
\[ (1\dots n) \otimes ( [1 + \tau + \dots + \tau^{n-1}] \phi^{i_1} \otimes \dots \otimes \phi^{i_n} ) + \dots, \]
where $\tau$ is the cyclic permutation of $1\mapsto 2 \mapsto \dots \mapsto n \mapsto 1$ and the $\dots$ at the end denote the $(n-1)! - 1$ remaining terms. In \eqref{BVlaplace_jako_selfcomposition}, we choose $\theta$ as
\[ \theta(1) = a, \,\theta(2) = b, \,\theta(k) = k-2 \text{ for } k>2. \]
The operator $\Delta$ then cuts the relabeled cycle $(1 \dots a \dots b \dots n)$ at $a$ and $b$ into two (possibly empty) cycles. To calculate $\Theta$, we need to find which terms contribute to the terms $x_\mathbf{b} \otimes \dots$, i.e. which cuts result in two cycles in the standard form \eqref{eq:standardb}. There are $2 k(n-2-k)$ such contributions\footnote{If $k=0$ or $n-2-k=0$, there are still two contributions; let us not mention this technicality again.} to each possible length of cycles, coming from term labeled $(a [1 \dots k] b [k+1 \dots n-2])$ and $(a [k+1 \dots n-2] b  [1 \dots k])$ for any $0 \le k \le n/2-1$. Using the symmetry of $(-1)^{\hdeg {\phi^a}}\omega^{ab}$, re-expressing the tensor $[1 + \tau + \dots + \tau^{n-1}] \phi^{i_1} \dots \phi^{i_n}$ using a cyclic permutation exchanging $a \leftrightarrow b$ and collecting the signs, one obtains that these contributions are equal. 

The factor $k(n-2-k)$ cancels with the normalisation $\prod_{i\ge 0} i^{b_i}$ of the map $\Theta$. The factor $2$ can be removed by expanding the possible values of $k$ to $n-2$; if $k = n-2-k$ and the two cycles are of the same length, this factor of $2$ instead cancels the $2!$ from the normalisation of $\Theta$. Together, we thus obtain
\[  \Delta (\phi^{i_1} \dots \phi^{i_n}) = \sum_{k = 0}^{n-2} \left[ (-1)^{(\hdeg{\phi^{i_2}} + \dots + \hdeg{\phi^{i_{k+1}}})\hdeg{\phi^{i_{k+2}}}  }\tilde\omega^{i_1 i_{k+2}} (\phi^{i_2} \dots \phi^{i_{k+1}}) (\phi^{i_{k+3}} \dots \phi^{i_n}) + \text{ cycl.} \right], \]
with the convention that an empty cycle (for $k=0$ or $k=n-2$) is replaced by $\xi$. The $n-1$ terms in $+\text{ cycl.}$ are obtained from $\tau^{m} \phi^{i_1}\otimes \dots \otimes \phi^{i_n}$ for $1\le m \le n-1$, i.e. they also contain the sign from permuting the graded covectors $\phi$. 
\smallskip

The bracket, defined in \eqref{definice_zavorky}, is computed similarly. Looking at 
\[ \{ (\phi^{i_1} \dots \phi^{i_{n_1}} ), (\phi^{j_1} \dots \phi^{j_{n_2}} )\} \]
for $n_1, n_2 > 1$, the only terms from the sum over decompositions which contribute to the cycle $(1 \dots n_1 + n_2 - 2)$  are those where $C_1$ is an interval $\{l, \dots, l+n_1 -1\} \mod (n_1+n_2-2)$. Moreover, for each such decomposition, only one permutation from the sum \eqref{eq:symnonstab} contributes. There are $n_1+n_2-2$ choices for $l$, which are all equal after the projection $[-]$; this cancels the normalization of $\Theta$. 

The case $n_2 =1$ is different: there is only one choice of a decomposition, and $n_1-1$ different permutations from \eqref{eq:symnonstab} contribute, namely the cyclic permutations of the interval $\{1, \dots, n_1-1\}$. 
\end{proof}

\begingroup
\let\clearpage\relax

\endgroup


\begin{thebibliography}
\\




\bibitem{BarannikovModopBV} S. Barannikov, ``Modular operads and Batalin-Vilkovisky geometry'', International Mathematics Research Notices, Oxford University Press, Vol. 2007 (2006). 
arXiv:1710.08442

\bibitem{BarannikovNCMI} S. Barannikov, ``Noncommutative Batalin–Vilkovisky Geometry and Matrix Integrals.'' Comptes Rendus Mathematique, vol. 348, no. 7–8, Apr. 2010, pp. 359–62. Crossref, \url{https://doi.org/10.1016/j.crma.2010.02.002}


\bibitem{Batalin1981}
I.~Batalin, G.~Vilkovisky, ``{Gauge algebra and quantization},''
  \href{http://dx.doi.org/10.1016/0370-2693(81)90205-7}{{\em Physics Letters B}
  {\bfseries 102} no.~1, (June, 1981) 27--31}.
  
\bibitem{BeilinsonDrinfeld}
A. Beilinson, V. Drinfeld, ``Chiral Algebras'', American Mathematical Society (2004).


\bibitem{BergerKaufmann} C. Berger, R. Kaufmann ``Trees, graphs and aggregates: a categorical perspective on combinatorial surface topology, geometry, and algebra'' (2022) arXiv:2201.10537. 

\bibitem{BorisovManin} D. Borisov,  and Y. Manin. ``Generalized operads and their inner cohomomorphisms.'' In Geometry and dynamics of groups and spaces, pp. 247-308. Birkhäuser Basel, (2007).

\bibitem{BraunMaunder}
C. Braun, J. Maunder, ``Minimal models of quantum homotopy Lie algebras via the BV-formalism'',
Journal of Mathematical Physics, Vol.59, Issue 6 (2018).
arXiv:1703.00082


\bibitem{CFL} K. Cieliebak, K. Fukaya, J. Latschev, ``Homological algebra related to surfaces with boundary'', (2015).
arXiv:1508.02741 


\bibitem{CostelloGwilliam} K. Costello, O. Gwilliam, ``Factorization Algebras in Quantum Field Theory'' (New Mathematical Monographs). Cambridge University Press. doi:10.1017/9781316678626 (2016). 


\bibitem{DoubekModularEnvelope} M. Doubek, ``The Modular Envelope of the Cyclic Operad Ass'', Appl. Categor. Struct. 25, 1187–1198 (2017).

\bibitem{DJMS} M. Doubek, B. Jur\v co, M. Markl, I. Sachs, ``Algebraic Structure of String Field Theory'', Lecture Notes in Physics 937, Springer (2020).

\bibitem{DJM} M. Doubek, B. Jur\v co, K. M\"unster, ``Modular operads and the quantum open-closed homotopy algebra'', Journal of High Energy Physics, Vol. 2015, No. 12 (2015). 
arXiv:1308.3223v2


\bibitem{Eisenbud} D. Eisenbud, ``Commutative Algebra: with a view toward algebraic geometry'' Vol. 150. Springer Science \& Business Media, (2013).


\bibitem{DJP} M. Doubek, B. Jur\v co, J. Pulmann, ``Quantum $L_\infty$ Algebras and the Homological Perturbation Lemma'', Communications in Mathematical Physics, Vol 367 (2017).
arXiv:1712.02696




\bibitem{GK} E. Getzler, M. M. Kapranov, ``Modular operads'', Compositio Mathematica, Vol. 110 (1998).
arXiv:dg-ga/9408003


\bibitem{gwilliam-thesis}
O.~Gwilliam, ``{Factorization algebras and free field theories},'' (2013).
\newblock \url{https://ncatlab.org/nlab/files/GwilliamThesis.pdf}.









\bibitem{JM} B. Jur\v co, K. M\"unster, ``Type II Superstring Field Theory:  Geometric Approach and Operadic Description'', J. High Energ. Phys., Vol. 2013, 4 (2013).
arXiv:1303.2323

\bibitem{Kaufmann-odd} R. M. Kaufmann, B. C. Ward, J. J. Zuniga, ``The odd origin of Gerstenhaber brackets, Batalin-Vilkovisky operators, and master equations'',
J. of Mathematical Physics, Vol. 56 Issue 10 (2015).
arXiv: 1208.5543

\bibitem{Kaufmann-Feynman} R. M. Kaufmann, B. C. Ward, Feynman categories, Astérisque 387 (2017), vii+161pp arxiv:1312.1269





\bibitem{markl2001loop}
M.~Markl, ``{Loop Homotopy Algebras in Closed String Field Theory},''
  \href{http://dx.doi.org/10.1007/PL00005575}{{\em Communications in
  Mathematical Physics} {\bfseries 221} no.~2, (July, 2001) 367--384},
  \href{http://arxiv.org/abs/hep-th/9711045}{{\ttfamily arXiv:hep-th/9711045}}.


\bibitem{M} M. Markl, ``Odd structures are odd'', Advances in Applied Clifford Algebras, Vol. 27 (2017).
arXiv:1603.03184


\bibitem{Lada} L. Peksová ``Modular operads with connected sum and Barannikov’s theory.'' Archivum Mathematicum, vol. 56 (2020), issue 5, pp. 287-300 \url{http://dx.doi.org/10.5817/AM2020-5-287}






\bibitem{Schwarz} A. Schwarz, ``Grassmannian and String Theory'', Communications in Mathematical Physics, Vol. 199 (1998).
arXiv:hep-th/9610122

\bibitem{SenZwiebach} A. Sen, B. Zwiebach. ``Quantum background independence of closed-string field theory. Nuclear Physics B 423.2-3 (1994): 580-630.
https://arxiv.org/abs/hep-th/9311009


\bibitem{zwiebach-closed}
B.~Zwiebach, ``{Closed string field theory: Quantum action and the
  Batalin-Vilkovisky master equation},''
  \href{http://dx.doi.org/10.1016/0550-3213(93)90388-6}{{\em Nuclear Physics B}
  {\bfseries 390} no.~1, (June, 1993) 33--152},
  \href{http://arxiv.org/abs/hep-th/9206084}{{\ttfamily arXiv:hep-th/9206084}}.

\bibitem{ZwiebachOC} B. Zwiebach,
``Oriented Open-Closed String Theory Revisited''
Annals of Physics, Volume 267, Issue 2 (1998).  	arXiv:hep-th/9705241


\end{thebibliography}
\end{document}